\documentclass[]{informs3nothing}  
\usepackage{amsmath}
\usepackage{amssymb}
\usepackage[numbers]{natbib}
\NatBibNumeric

	\newcounter{theorem}
	\newcounter{corollary}
	\newcounter{lemma}
	\newcounter{definition}
	\newtheorem{mythm}[theorem]{Theorem}
	\newtheorem{mycor}[corollary]{Corollary}
	\newtheorem{mydef}[definition]{Definition}
	\newtheorem{mylem}[lemma]{Lemma}

\newcommand{\stackeq}[1]{
	\stackrel{\mathrm{(#1)}}{=}
}
\newcommand{\vx}{\mathbf{x}}
\newcommand{\vu}{\mathbf{u}}
\renewcommand{\P}{\mathbb{P}}
\newcommand{\E}{\mathbb{E}}

\TITLE{A Multi-Urn Model for Network Search}
\RUNTITLE{Multi-Urn Model for Network Search}

\RUNAUTHOR{Marks and Zaman}
\ARTICLEAUTHORS{%
\AUTHOR{Christopher E Marks}
\AFF{%
Operations Research Center, Massachusetts Institute of Technology\\
Charles Stark Draper Laboratory \\
555 Technology Square \\
Cambridge, MA  02139 \\
\EMAIL{cemarks@mit.edu}
}
\AUTHOR{Tauhid Zaman}
\AFF{%
Department of Operations Management, Sloan School of Management \\ Massachusetts
Institute of Technology \\ 
77 Massachusetts Ave. \\
Cambridge, MA 02139 \\ 
\EMAIL{zlisto@mit.edu}
}
}

\ABSTRACT{
We consider the problem of finding a specific target individual hiding in a social network.  We propose a method for network vertex search that looks for the target vertex by sequentially examining the neighbors of a set of ``known'' vertices  to which the target vertex may have connected.  The objective is to find the target vertex as quickly as possible from amongst the neighbors of the known vertices.  We model this type of search as successively drawing marbles from a set of urns, where each urn represents one of the known vertices and the marbles in each urn represent the respective vertex's neighbors.   Using a dynamic programming approach, we analyze this model and show that there is always an optimal ``block'' policy, in which all of the neighbors of a known vertex are examined before moving on to another vertex.  Surprisingly, this block policy result holds for arbitrary dependencies in the connection probabilities of the target vertex and known vertices.  Furthermore, we give precise characterizations of the optimal block policy in two specific cases: (1) when the connections between the known vertices and the target vertex are  independent, and (2) when the target vertex is connected to at most one known vertex.  Finally, we provide some general monotonicity properties and discuss the relevance of our findings in social media and other applications.
}

\begin{document}

\maketitle

\section{Introduction} \label{sec: intro}

In this paper, we examine the problem of searching a social network for a particular target individual by sequentially examining the neighbors of other known users.
Social media applications enable users to connect with each other, forming social networks.  For different reasons which we will discuss, one may wish to find a target individual in the social network.  One may have prior knowledge that the target individual is connected with a set of known users, and so the most logical place to begin searching is the neighbors of these known users.  If querying each of these neighbors incurs some sort of cost, then the goal would be to find the target with as few queries as possible.

For example, suppose Mary is searching a social media application for an account belonging to John, an old friend from school with whom she has lost contact.  From what she knows about John, Mary might be able to develop a list of accounts she knows about within the social media application to which John's account might be connected.  For example, she might recall that John was good friends with Matt, who has a social media account that is known to Mary.  She also might remember John was active in a certain charity, which also maintains a social media account known to Mary.  

After developing such a list, Mary could sequentially explore each account's connections, but doing so could take a substantial amount of time.  In order to find John's account quickly (assuming John has an account), Mary might devise a search strategy.  For example, she might start by looking at accounts she feels are the most likely to be connected with John's account.  Alternatively,  she might start by looking at accounts with fewer connections, because her goal is to find John's account while minimizing the time spent exploring.

In this hypothetical scenario, what Mary is doing is an example of a network \emph{vertex search} in which the sought object, or \emph{target} might be found by examining the neighbors of a finite set of known vertices.  Once the search target is found, the search typically terminates.  Each known vertex $i$ might have a different degree, $N_{i}$, requiring a different number of search queries to exhaustively search.  From our scenario, we also consider that each known vertex $i$ might have a different probability of being connected to the target vertex, which we denote as $\varphi_{i}$.

In this paper we present a probabilistic ``multi-urn'' model for searches of this nature in which we represent each known vertex as an urn containing a finite number of marbles.  Each marble in an urn represents one of the respective vertex's neighbors.  The search consists of successively drawing and and examining individual marbles from the urns with the goal of finding a red marble, representing the search target, in the fewest number of draws.  Figure \ref{fig: urn_network} depicts the multi-urn model for a network search on three known vertices.  In this example, each of the known vertices is connected to the search target, so each urn contains a single red marble.  Additionally, each urn contains a number of blue marbles that represent the other neighbors belonging to each respective vertex. 
  Employing a dynamic programming framework, we provide some insight into the optimal search policy in this general model and under certain conditions.

\begin{figure}[!htb]
\centering
\fbox%
{
\includegraphics[width=0.4\textwidth]{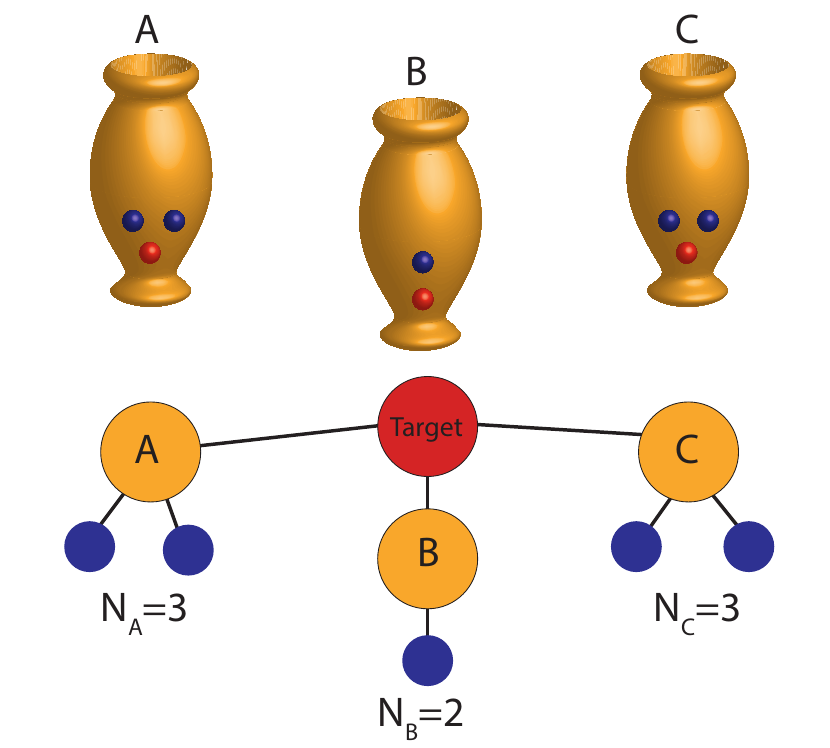}
}
\caption{Network search representation as a multi-urn model.} \label{fig: urn_network}
\end{figure}

\subsection{Application}

Our model applies to any scenario where one must sequentially search for the target amongst a set of entities which are separated into different clusters.  In our vertex search problem, the entities are vertices and the clusters are neighborhoods of the known vertices.   Our main motivation for this model is in social media applications where many times the goal is to find  users who harass others, incite violence, or engage in other dangerous behaviors.  Twitter has been suspending large numbers of users, many of which support or engage in violent extremism, from its micro-blogging application for violating the site's published rules \cite{usatoday}.  The challenge is that these users can simply create a new account each time one is suspended.  However, from historical data Twitter could predict the accounts to which the suspended user is likely to connect.  Using this information, Twitter administrators could then apply our optimality criteria to efficiently locate new accounts belonging to suspended users.  

 There are other scenarios where this model can apply.  For instance, for law enforcement and intelligence applications, the search entities could be suspects in a crime and the clusters could be geographical locations.  Or if one is examining a dump of emails from a suspect's server, one may be looking for an incriminating email, so the entities are emails and the clusters could be recipients of the emails.  In both of these examples, the process of querying the entities requires a non-trivial amount of resources (interviewing a suspect, reading an email),  so it is important to find the target as quickly as possible.  Using an optimal or near-optimal search strategy is therefore crucial in these examples.

\subsection{Our Contributions}

In this paper we develop a multi-urn search model, a new and useful methodology for analyzing searches similar to the network vertex search problem we proposed in the introduction.  We employ a dynamic programming framework to analyze this model and  provide theoretical results based on the nature of the cost function.

In particular, we show that in this type of search problem there always exists an optimal policy, i.e., one that minimizes the expected number of marbles drawn before finding a red marble, in which  the marbles in each urn are exhausted before moving on to the next urn (Theorem \ref{thm: blockpolicy}).  We refer to a policy that meets this criterion as a \emph{block policy}, and show that this result holds irrespective of dependencies among the urns.  This result is surprising because the presence of the red ball in the urns could have arbitrary correlations.  There could be positive correlations, where if the ball is in one urn, it is more likely to be in another urn.  Or the correlations could be negative, so if a ball is in one urn, it is less likely to be in another urn.  Nonetheless, our result shows that no matter what the dependency between the urns, a block policy is always an optimal policy to find the red ball.  

Building on this finding, we provide optimality conditions that enable immediate determination of an optimal policy in two specific cases: 
\begin{itemize}
    	\item Each urn contains a red marble independent of other urns (Theorem \ref{thm: independence_optimality}).  This case corresponds to the assumption that each known vertex is connected to the search target independent of other known vertices' connections.
    	\item  There is at most one red marble among all of the urns (Theorem \ref{thm: single_marble_optimality}).  This case corresponds to the assumption that the search target is connected to at most one known vertex.
    \end{itemize}    

Finally, we provide detailed analysis of the system dynamics of the multi-urn search model, leading to insight into the intuition behind our findings.  We provide monotonicity properties on the evolution of certain probabilities as marbles are successively drawn (Theorem \ref{thm: monotonicity}) and establish a useful bound on how much the probability of drawing a red marble from a certain urn can change between two successive stages (Theorem \ref{thm: probrate}).
\subsection{Previous Work}

Much of the work that has been done in the context of network search is focused on finding \emph{relevant} vertices in a large scale network.  Google's PageRank algorithm is probably the most well known example of these methods, of which many adaptations and generalizations exist \cite{pagerank-survey}.  Our work looks at an essentially different type of network search: one of finding a \emph{specific} vertex in a network, presumably identifiable by certain features, by investigating network neighborhoods in which the vertex is likely to appear.

The network vertex search problem we have proposed could instead be formulated as a multi-arm bandit problem.   \citet{bandit-survey} provide a broad survey of many variations of the multi-arm bandit problem and their respective applications.  These problems are typically likened to a gambler who has a choice of playing from a set of slot machines.  At each discrete stage in the process the gambler selects and plays a slot machine for a certain cost and receives a stochastic reward from an unknown distribution.  The more times the gambler plays a particular machine, the more he is able to learn about its reward distribution.

In the multi-arm bandit setting, the gambler would not want to spend too much money playing low-payout slot machines just to learn their reward distribution.  This quandary is the fundamental tradeoff between \emph{exploration} and \emph{exploitation}, which is inherent in multi-arm bandit problems.  In order to make money, the gambler wants to play only the highest-payout slot machine.  However, he never really knows the true distributions of any of the machines.  As a result, optimal multi-arm bandit policies often include a balance of exploratory actions, in which decisions are made for the sole purpose of observing outcomes, and exploitative actions, in which decisions are made to optimize the outcomes based on what has been learned.

The multi-arm bandit problem objective is often characterized as the minimization of \emph{regret}, which is essentially the difference in expectation between what the gambler earns and what he would have earned by playing the best machine.  \citet{lai} provide a very well-known method for constructing adaptive multi-arm bandit policies using upper confidence bounds, for which regret grows proportional to the logarithm of the number of plays in the limit.   \citet{auer2002finite} show that this same bound on regret is also achievable in finite time.

The multi-urn search model we present could be cast in the context of a finite time multi-arm bandit problem, but there are a few notable differences.  Our objective, to find the search target as quickly as possible, does not immediately cast itself as minimizing regret.  
\citet{gittins1979bandit} overcomes this difficulty by augmenting the state space in the multi-arm bandit formulation with a ``success'' state, from which no additional costs or rewards are incurred.  Building on this adaptation, \citeauthor{gittins1979bandit} describes a class of search problems that are very similar to our network search problem, and characterizes the optimal search policy based on his well-known dynamic allocation index \cite{gittins1974}.

Our search problem differs from \cite{gittins1979bandit}, however, in that each vertex has a fixed, finite, and known number of neighbors.  In essence, we assume the reward distribution of each slot machine is known, and we only allow a fixed, finite number of plays on each machine.  Unlike the bandit approach, the outcomes of successive marble draws from a single urn are not assumed to be independent observations from an unknown distribution.  Instead, our model uses a known distribution on each machine but limits each machine to allowing at most a single win.  Furthermore, in our approach we allow for dependencies between the urns, whereas the multi-arm bandit approach typically assumes each slot machine's outcomes are independent of the others.

In spite of these differences, the dynamic allocation index applied in the class of search models proposed by \citet{gittins1979bandit} has many similarities to our development.  The system dynamics in both cases are governed by Bayesian probability updates.  We show that in at least two cases the optimal policy can be characterized by a priority index, which is derived directly from the system dynamics and can be interpreted as the expected rewards of decisions.  \citet{gittins1979bandit} also mentions monotonicity properties of his dynamic allocation index in the context of search that are similar to the monotonicity properties we derive.  Our method for proving optimality uses similar logic to the proofs given in \cite{frostig}, which are based on the original development by \citet{gittins1974}.

Like multi-arm bandit problems, urn models have been applied in many contexts, including discrete decision processes.  The P\'olya urn process is a well-known construct using urns that has been adapted and used in many applications \cite{mahmoud2008polya}.  This process generally consists of one or more urns, each containing certain numbers of marbles of different colors.  At each stage in the process a marble is randomly drawn from an urn and its color observed.  This color then dictates an action involving placing one or more marbles of certain colors into certain urns.  

\citet{polya-medical} provides a specific adaptation the P\'olya urn process to the problem of conducting medical trials in a way that is meant to exploit the use of treatments that have shown positive results in the past, which is very similar to multi-arm bandit models applied in the same context.  The P\'olya urn process has also been used as  a preferential attachment model in the formation of networks \cite{Chung2003}.  This application can be useful in considering how links form in social networks, and is similar to our problem.  We assume, however, that the links are already present in the network and are instead interested in finding the optimal way to investigate these existing links.

\citet{downey2006probabilistic} employ a multi-urn model that is very similar to ours but serves a different purpose: unsupervised information extraction.  The model these authors propose uses urns to represent different collections of documents.  Marbles drawn from the urns represent specific documents, from which specific labels are extracted.  The objective of the model is to learn which labels are the correct, or ``target'' labels, and which labels are erroneous extractions.

The urn model proposed in \cite{downey2006probabilistic} differs substantially from ours in its objective.  \citeauthor{downey2006probabilistic} have the objective of learning model parameters and, in the unsupervised case, learning which labels are correct.  In the urn model we present, we assume the probability distributions and the target labels are known a priori, and we aim to to find a target marble as efficiently as possible.

Our network search problem is also related to the problem of mutual information maximization.  If our goal was mutual information maximization, we would not necessarily focus our search effort on trying to find the target vertex.  Instead, we would examine the places that would give us the most information about where the target is likely to be.  This is similar to the goal of exploration in the multi-arm bandit problem.
\citet{chen2015sequential} analyze a sequential information maximization problem that parallels our development, using a dynamic programming approach and giving bounds on the performance of the greedy approach.  The problem the authors propose involves learning about the distribution of an unknown parameter of interest by sequentially observing other variables.  Each observation provides some information about the unknown parameter, and the objective is to maximize the total information gained in a fixed number of observations.  

Our multi-urn search model departs most substantially from the development in \cite{chen2015sequential} by imposing additional constraints and dynamics in the way observations are made.   In our model, the urns are depleted over time, changing the amount of information contained in each successive marble drawn in predictable, but sometimes unintuitive ways.  Our main contributions are the characterizations of  optimal search policies under various probability models, which come directly from analysis of the dynamics inherent in our multi-urn search model.

Finally, recent work in scheduling and inspection policies employ similar dynamic programming approaches to characterize optimal policies.  \citet{levi2016scheduling} use dynamic programming to find policies that optimally allocate resources between information gathering and task execution.  This class of models provides a natural extension to our network search problem.  While we assume a probability model on a set of known vertices, using this approach we could attempt to find the optimal balance between the time spent learning a probability model on a set of known vertices and the time spent executing the search on the current known vertex set.

\section{Multi-urn Search Model} \label{sec: model}

We return to the context of network vertex search as presented in the introduction.  Let $\mathcal{V}$ be the set of vertices to search, and let $N_{i}$ be the number of search queries required to exhaustively search vertex $i \in \mathcal{V}$.  We assume that the neighbors of each vertex are queried in a random order, so that each individual neighbor query of a particular vertex is equally likely to be the search target, given that the target is connected to the queried vertex. 
 
Under these assumptions, we can represent this search problem as an experiment involving randomly drawing marbles from a set of urns, where each urn represents a known vertex in the network.  Each marble in urn $i \in \mathcal{V}$ represents a neighbor of vertex $i$.  The degree of vertex $i$ is $N_{i}$, so urn $i$ initially has $N_{i}$ marbles.  With probability $\varphi_{i}>0$, exactly one of the $N_{i}$ marbles in urn $i$ is red, indicating that known vertex $i$ is connected to the target vertex.  Otherwise, all marbles in all urns are blue.

Querying a random neighbor of vertex $i$ in search of the target is analogous to drawing a random marble from urn $i$ and observing its color.  If the marble is red, the target vertex has been located.  If the marble is blue, the target has not been found and the search continues with the remaining marbles.  Note that blue marbles are \emph{not} put back into the urns; once they are drawn they are discarded.
Just as Mary desires to find her old friend John with as few searches as possible, the goal in this experiment is to minimize the number of blue marbles drawn before finding a red one.  

We now more completely specify the probability model that accounts for how the target vertex might be connected to the set of known vertices, i.e., how red marbles might be distributed among the urns.  Let $A_{i}$ be the event that the target vertex is connected to vertex $i$.  We have already defined
\[
\varphi_{i}=\P(A_{i}).
\]
More generally, we let
\[
\varphi_{U}=\P\left(\bigcap_{i\in U}A_{i}\right)
\]
be the probability that the target vertex is connected to all vertices in set $U \subseteq \mathcal{V}$.  If we were to assume that the target would connect to the members of $U$ independently, then $\varphi_{U}=\prod_{i \in U}\varphi_{i}$.  In general, the connections might not be independent.  For example, Mary might think that if John connected with a certain musician he liked, he might be more likely to connect to other, similar musicians.   In other cases, a connection to a particular vertex might imply a decrease in the probability of connection to another vertex.

In our urn model, we assume a known probability $\varphi_{U}$ for all subsets $\{U: U \subseteq \mathcal{V}\}$, which fully specifies a probability model on the locations of the red marbles among the urns.  It allows for arbitrary correlations between urns, so that the presence of a red marble in one urn (or subset of urns) can have a positive or negative correlation with the presence of a red marble in another urn (or another subset of urns). 

We note now that the empty set $\emptyset \in \{U:U\subseteq \mathcal{V}\}$.  By convention, we set $\bigcap_{i\in\emptyset}A_{i}=\Omega$, so that $\varphi_{\emptyset}=1$. This term is implicitly included in summations over all subsets expressed in this paper.  For example, the summation 
\[\sum_{U\subseteq{\mathcal{V}}}(-1)^{|U|}\varphi_{U}
\]
includes a ``1'' corresponding to the case in which $U=\emptyset$.

Given this set of probabilities, the probability of any specific outcome of marble locations, or vertex connections, can be determined using the well-known inclusion-exclusion formula.  For example, suppose we are interested in the probability that the marble is located in all of the urns in set $U$ and no other urns.  This event can be written as $\left(\bigcap_{i \in U}A_{i}\right)\cap\left(\bigcap_{j\in \mathcal{V} \setminus U}A_{j}^{c}\right)$, with
 \begin{equation}
 \P\left(\left(\bigcap_{i \in U}A_{i}\right)\cap\left(\bigcap_{j\in \mathcal{V} \setminus U}A_{j}^{c}\right)\right)
 = \sum_{S: S \subseteq \mathcal{V}, \ U \subseteq S}
 (-1)^{|S|-|U|} \varphi_{S} \geq 0. \label{eq: inclusion-exclusion}
 \end{equation}

Throughout this paper, we refer to the type of search described in this section as a \emph{multi-urn search problem} which we now more formally define.  
\begin{mydef} \label{def: multiurn}
A \emph{multi-urn search problem} is a search problem that can be modeled as sequentially drawing marbles from a set of urns, $\mathcal{V}$, where
\begin{enumerate}
\item The objective of the searcher is to find a red marble with as few draws as possible.
\item Each urn $i \in \mathcal{V}$ contains at most a single red marble.  Otherwise, all marbles are blue.
\item Each urn $i \in \mathcal{V}$ contains a known number of marbles, $N_{i}$.
\item For each subset of urns $U \subseteq \mathcal{V}$, the probability that a red marble is present in all urns in $U$ is $\varphi_{U}$.  Additionally, we set $\varphi_{\emptyset}=1$.
\end{enumerate}
\end{mydef}

The network vertex search problem we used to motivate this model can be characterized as a multi-urn search problem, but this model might have other useful applications as well.  For this reason, in the remainder of this paper we provide all analyses and results in the multi-urn search context, using the language of ``urns'' and ``marbles,'' though we could immediately recover our original context by substituting ``known vertices'' and ``neighbors,'' respectively.  


\subsection{Dynamic Programming Framework}

In the search model we have defined, the decisions are carried out sequentially in discrete stages. We now take a dynamic programming approach \cite{bertsekas-dp} to framing this problem.  

We model the search process as a discrete dynamic system of the form
\[
\mathbf{x}(t+1)=f(\mathbf{x}(t),u(t),w(\vx(t),u(t))),
\]
where $t = 0,1\ldots $ is the \emph{stage} of the search, which we equate to the total number of marbles already drawn from the urns.  The system state, $\vx(t)$, is a record of the total number of marbles drawn from each urn, which sufficiently characterizes the system at stage $t$.  Parameter $u(t)$ is the decision made, or urn selected, at stage $t$, and $w(\vx(t),u(t))$ is a binary stochastic input that indicates whether a red marble is drawn from urn $u(t) \in \mathcal{V}$ in state $\vx(t)$.  

If a red marble is drawn at stage $t$, then $w(\vx(t),u(t))=1$ and the search terminates.  Otherwise, $w(\vx(t),u(t))=0$ and the search continues.  Letting $x_{i}(t)$ be the number of marbles that have been removed from urn $i$ at time $t$, we can explicitly define the state vector
\[
\mathbf{x}(t)=(x_{1}(t),x_{2}(t),\ldots,x_{|\mathcal{V}|}).
\]
If a blue marble is drawn from urn $u(t)$ in state $\vx(t)$, the state transition function is:
\[
f(\mathbf{x}(t),u(t),0)=\vx(t)+\mathbf{e}_{u(t)},
\]
where $\mathbf{e}_{i}$ is the $i$th unit vector.  If a red marble is drawn at any stage, the search terminates.

Our dynamic programming model consists of at most $N+1$ stages, where $N = \sum_{i \in \mathcal{V}} N_{i}$ is the total number of marbles summed over all of the urns, and is finite.

We define a valid  \emph{policy} $\vu$ as a sequence of decisions 
$
(u(0),u(1),\ldots,u(N-1)),
$
where $u(t) \in \mathcal{V}$ for $t=0,1,\ldots,N-1$, \emph{and} for which
\[
|\{t: \; u(t)=i\}| = N_{i} \qquad \forall \ i \in \mathcal{V}.
\]
This final condition ensures that the policy will eventually exhaust each urn, as long as a red marble is not found, while at the same time never attempting to draw marbles from an empty urn.  A searcher executing a valid policy draws a marble from urn $u(t)$ at each stage $t$ until either the target marble is found or there are no marbles remaining in any of the urns, in which case the entire policy has been executed.

We note that in this dynamic programming model there is no benefit in making policy decisions during search execution.  At each stage the searcher either draws a red marble, in which case she stops looking, or draws a blue marble and keeps searching.  A valid policy provides an ordering of urn queries that is essentially conditioned on not drawing a red marble, which can be considered a deterministic process governed by our simple state transition function.  The expected search outcomes for such a policy can be analyzed and compared to those of other policies a priori.

\subsubsection{Stage $t$ Probability of Drawing a Red Marble}

Building on our dynamic programming modeling assumptions, we now develop the probability distribution associated with $w(\vx(t),u(t))$.  Recall that this function indicates whether a red marble is drawn in stage $t$: $w(\vx(t),u(t))=1$ implies a red marble is drawn from urn $u(t)$ at stage $t$, while $w(\vx(t),u(t))=0$ implies a blue marble is drawn from $u(t)$ at stage $t$.  

In determining the probability distribution of $w(\vx(t),u(t))$, it is important to remember that in order to arrive in stage $t$ while executing policy $\vu%
$, the preceding queries $u(0),u(1),\ldots,u(t-1)$ would have been executed \emph{without drawing a red marble}, so that the system arrives in state $\vx(t)$.  For simplicity of notation, we condition an event on state $\vx(t)$ to imply that state $\vx(t)$ has been reached without having drawn a red marble.  For example, $\P(A_{i}\mid\vx(t))$ represents the probability urn $i$ contains a red marble, given queries $u(0),u(1),\ldots,u(t-1)$ have been executed without drawing a red marble.  

Using the multiplication rule, we can write the probability
\begin{equation}
\P(w(\vx(t),u(t))=1)=\left(\frac{1}{N_{u(t)}-x_{u(t)}(t)}\right)\P(A_{u(t)}\mid\vx(t)), \label{eq: w1}
\end{equation}
which is the probability of drawing a red marble from the $N_{u(t)}-x_{u(t)}(t)$ marbles remaining in urn $u(t)$, given there is a red marble in $u(t)$, multiplied by the probability urn $u(t)$ contains a red marble given the system has arrived at state $\vx(t)$.  

The complementary probability can be written using the law of total probability:
\begin{align}
\P(w(\vx(t),u(t))=0)&=\left(1-\frac{1}{N_{u(t)}-x_{u(t)}(t)}\right)\P(A_{u(t)}\mid\vx(t))
+\P(A^{c}_{u(t)}\mid\vx(t)) \nonumber \\
&=1-\left(\frac{1}{N_{u(t)}-x_{u(t)}(t)}\right)\P(A_{u(t)}\mid\vx(t)) \label{eq: w0}
\end{align}

\subsubsection{Stage $t$ Urn Probabilities}

In this process, we have assumed a fully specified initial probability model on the urns, i.e., for any subset $U \subseteq \mathcal{V}$, the probability that a red marble is present in all of the urns, $\varphi_{U}$, is known.   This probability model can be thought of as a Bayesian prior, a quantification of the searcher's beliefs on where a red marble might be found.

However, these probabilities are not static.  After drawing a marble from an urn, the probabilities change as a result of the new information.  If the marble drawn is red, then the probability that a red marble existed in the queried urn becomes 1.  Likewise, if the marble drawn is blue, then the probability that a red marble can be found in the queried urn decreases as a function of the number of marbles remaining in the urn and the current urn probability.

As long as a red marble is not found, the evolution of urn probabilities over the course of the search is completely determined by the initial probability model and the search policy.  We now provide a general expression for updated urn probabilities at stage $t$.

\begin{mythm} \label{thm: genprob}
{\bf (Urn Probabilities).}
In a multi-urn search problem over a set of urns $\mathcal{V}$, suppose a red marble is not found in the first $t$ queries when executing a valid policy $\vu=(u(0),\ldots,u(N-1))$.  Then, for any subset of urns $U \subseteq \mathcal{V}$, the probability of a red marble being in all of the urns in $U$ at stage $t$ is given by:
\[
\P\left(\bigcap_{i\in U} A_{i}\;\middle|\;\vx(t)\right) = 
\frac{
	\left[
		\prod_{i \in U}
		\left(
			1-\frac{
					x_{i}(t)
				}{
					N_{i}
				}
		\right)
	\right]
	\sum_{
		\{S \subseteq \mathcal{V}: \; S \supseteq U\}
	}
	(-1)^{|S|-|U|}\varphi_{S}
	\prod_{
		j \in S \setminus U
	}
	\frac{
		x_{j}(t)
	}{
		N_{j}
	}
}{
	\sum_{S \subseteq \mathcal{V}}
	(-1)^{|S|}
	\varphi_{S} 
	\prod_{j \in S} 
	\frac{
		x_{j}(t)
	}{
		N_{j}
	}
}.
\]
\end{mythm}

\begin{proof}{Proof.}
See Section \ref{proof: genprob}.
\end{proof}
Substituting the result in Theorem \ref{thm: genprob} into Equations \eqref{eq: w1} and \eqref{eq: w0} gives us the following corollary.
\begin{mycor} \label{cor: drawprob}
The probability distribution of $w(\vx(t),u(t))$, indicating whether a red marble is drawn at stage $t$ is,
\[
\P(w(\vx(t),u(t))=k)=\begin{cases}
\frac{
	\sum_{S \subseteq \mathcal{V}}
	(-1)^{|S|}
	\varphi_{S} 
	\prod_{i \in S} 
	\frac{
		x_{i}(t+1)
	}{
		N_{i}
	}
}{
	\sum_{S \subseteq \mathcal{V}}
	(-1)^{|S|}
	\varphi_{S} 
	\prod_{j \in S} 
	\frac{
		x_{i}(t)
	}{
		N_{i}
	}
},
 & k=0
\\
\left(\frac{1}{N_{u(t)}}\right)
\frac{
	\sum_{
		\{S \subseteq \mathcal{V}: \; u(t) \in S\}
	}
	(-1)^{|S|-1}\varphi_{S}
	\prod_{
		i \in S \setminus \{u(t)\}
	}
	\frac{
		x_{i}(t)
	}{
		N_{i}
	}
}{
	\sum_{S \subseteq \mathcal{V}}
	(-1)^{|S|}
	\varphi_{S} 
	\prod_{i \in S} 
	\frac{
		x_{i}(t)
	}{
		N_{i}
	}
},
 & k=1.
\end{cases}
\] 
\end{mycor}

Theorem \ref{thm: genprob} provides a few important insights.  First, we can see that the urn probabilities at any stage depend only on the numbers of marbles drawn from all of the urns, and not the order in which they were drawn.  This \emph{path-independence} property of the stage $t$ urn probabilities is somewhat intuitive: a specific set of marbles drawn gives us a fixed amount of information irrespective of the order in which we inspect the marbles.  We will make use of the path-independence property in the proofs for Theorems \ref{thm: blockpolicy}, \ref{thm: independence_optimality}, and \ref{thm: single_marble_optimality}.

Another observation is that the form of the probability expression is similar to the well-known inclusion-exclusion formula for computing probabilities of unions of events.  In fact, this probability is an application of the principle of inclusion-exclusion applied in conjunction with Bayesian updates.  In Lemma \ref{lem: problemma}, we explicitly  define the events that are characterized by the inclusion-exclusion formulas in Theorem \ref{thm: genprob}.

\subsection{Costs}

In many network search applications, the cost of finding and examining a (random) neighbor of a known vertex is primarily the time consumed in executing the query and reviewing the results to determine whether the neighbor is the search target.  Because we have no reason to believe this time-cost would be different for different vertex-neighbor queries, we assume in our model that the cost of drawing a marble is the same for all urns.  The goal of the searcher is simply to minimize the number of blue marbles drawn, or vertex-neighbor queries executed, before finding the search target.  

We therefore define the cost function at stage $t$, 
\[
g(t)=\begin{cases} 1 & w(\vx(t),u(t))= 0 \\
0 & \mathrm{otherwise},
\end{cases}
\]
which applies a unit cost for every blue marble drawn.  Because this quantity is stochastic, we set as our  objective the minimization of expected total cost.  Letting random variable $C= \sum_{t=0}^{N} g(t)$, we aim to find the optimal policy $\vu$ to solve the following optimization problem:
\[
\mathrm{minimize}_{\vu} \ \E\left[C\right].
\]
Because $C \in \{0,1,\ldots,N\}$ almost surely, we can write
\begin{align}
\E\left[C\right] & = \sum_{k=0}^{N-1} \P\left(C > k\right) \nonumber \\
& = \sum_{k=0}^{N-1} \prod_{t=0}^{k}\P(w(\vx(t),u(t))=0). \label{eq: simplecost}
\end{align}
The product in this summation, $\prod_{t=0}^{k}\P(w(\vx(t),u(t))=0)$, is exactly the probability of making it to stage $k+1$ without having found a red marble.  From Corollary \ref{cor: drawprob} we can find an expression for this probability.
\begin{mycor} \label{cor: cost}
Given a multi-urn search problem and valid policy $\vu$, the probability of arriving in stage $k+1$ without having found a red marble is
\begin{align*}
\P(C>k) &=\prod_{t=0}^{k} \P(w(\vx(t),u(t))=0) \\
& = 	\sum_{S \subseteq \mathcal{V}}
	(-1)^{|S|}
	\varphi_{S} 
	\prod_{i \in S} 
	\frac{
		x_{i}(k+1)
	}{
		N_{i}
	}.
\end{align*}
\end{mycor}
We can therefore rewrite the cost function,
\begin{equation}
\E\left[C\right]  = 
\sum_{t=0}^{N-1}
\sum_{S \subseteq \mathcal{V}}
	(-1)^{|S|}
	\varphi_{S} 
	\prod_{i \in S} 
	\frac{
		x_{i}(t+1)
	}{
		N_{i}
	}.
	\label{eq: fullcost}
\end{equation}
Substituting the probability from Corollary \ref{cor: cost} into Corollary \ref{cor: drawprob} reveals an interesting property of the dynamics of this system:
\begin{equation}
\P\left(w(\vx(t),u(t))=0\right)=
\frac{\P(C > t)}{\P(C>t-1)}. \label{eq: cost_property}
\end{equation}

\section{Key Results} \label{sec: findings} 

The cost function in equation \eqref{eq: fullcost} is nonlinear and nonconvex.  Additionally, for a solution to be feasible, the values for $x_{i}(t)$, $i=1,\ldots,|\mathcal{V}|$, $t=0,\ldots,N$ must be constrained to correspond to stages reached by a valid policy.  Nonlinear, non-convex constrained optimization is difficult in general.  However, the structure of the cost function enables us to derive some useful results that characterize the optimal solution in general, and provide necessary and sufficient conditions for optimality in some specific cases.

\subsection{Block Policy Optimality} \label{subsec: blockpolicy}

We now provide our primary general result, in which we  give a characterization of an optimal search policy in the multi-urn search problem.  We begin with a definition.

\begin{mydef} \label{def: blockpolicy}
A \emph{block policy} is a valid policy $\mathbf{u}_{B}$ in which each urn is queried exhaustively prior to querying another urn.  A block policy can be specified as a sequence of urns $\mathbf{u}_{B}=(v^{1},v^{2},\ldots,v^{|\mathcal{V}|}), \ v^{i} \in \mathcal{V}$, implying
\[
u(t) = v^{i}
, 
\quad
	\sum_{j=1}^{i-1} N_{j} 
\leq t < 
	\sum_{j=1}^{i} N_{j}
.
\]
\end{mydef} 
This definition can be used to characterize the optimal policy, which we now formally state.

\begin{mythm} \label{thm: blockpolicy}
{\bf (Block Policy Optimality).}
Given a multi-urn search problem in which the objective is to minimize the expected number of searches required to find a red marble, an optimal search policy exists that is a block policy.
\end{mythm}

\begin{proof}{Proof.}
See Section \ref{proof: blockpolicy}.
\end{proof}

This result says that an optimal policy for the multi-urn search problem can be characterized by a sequence
of urns.  Once this is specified, one then simply searches each urn until it is out of balls or a red ball is found.
The surprising part of this result is that this block policy optimality holds for arbitrary correlations in the a priori
connection probabilities.  For instance, there can be a negative correlation between two urns, where if the red ball is more likely
to be in one urn, it is less likely to be in another.  In this case one may intuitively expect that after querying an urn many times
and not finding a red ball, at some point it might be advantageous to search another urn which has a negative correlation with the queried urn.  However, our result says that it is optimal to continue querying the current urn until it is exhausted.  

While Theorem \ref{thm: blockpolicy} allows for optimal policies that are not block policies, constructing such a case requires initial conditions that include probabilities that are zero.  If $\varphi_{\{i,j\}}>0$ for all pairs $\{i,j\} \subset \mathcal{V}$ (as in the case of independent urns), then only block policies can be optimal policies.  This result follows from the proof of Theorem \ref{thm: blockpolicy} (Section \ref{proof: blockpolicy}): observe that this condition implies that function $h(t)$ in equation \eqref{eq: ht} is strictly increasing in $t$, creating a contradiction in equation \eqref{eq: hineq}.  

We have shown that for mutli-urn search problems, a block policy is optimal, but we have not yet specified
what the block policy is.  In general it can be difficult under arbitrary correlation structures to find the optimal policy.
However, under certain assumptions on the urn probability model, explicit necessary and sufficient optimality conditions can be found.
We examine these conditions next.

\subsection{Independent Urns} \label{subsec: independence}

We now consider the special case in which the red marbles are assumed to be independently present in each of the urns, so that the presence of a red marble in any urn (or group of urns) does not affect the probability of a red marble being present in any other urn (or group of urns).  This probabilistic independence can be formalized mathematically.

\begin{mydef} \label{def: independence}
An \emph{independent} multi-urn search problem is a multi-urn search problem in which, for any subset of urns, $U \subseteq \mathcal{V}$, 
\[
\varphi_{U}=%
\prod_{i \in U} \varphi_{i}.
\]
\end{mydef}
Intuitively, this independence property should be maintained throughout the search process for any search policy, as we now show explicitly.

\begin{mythm} \label{thm: independence}
{\bf (Independent Urn Probabilities).}
Given an independent multi-urn search problem,  then for any policy $\mathbf{u}$, at any stage $t$, the independence property is maintained so that
\[
\P\left(\bigcap_{i \in U} A_{i} \; \middle| \; \vx(t)\right) = \prod_{i \in U} \P(A_{i} | \vx(t)).
\]
\end{mythm}

\begin{proof}{Proof.}
See Section \ref{proof: independence}
\end{proof}
It follows from the result in Theorem \ref{thm: independence} that the probability of finding a red marble in stage $t$ is only a function of the initial conditions and number of times $u(t)$ has been queried in the past.  The number of marbles that have previously been drawn from other urns $i \neq u(t)$ do not affect $\P(w(\vx(t),u(t))=1)$.

Because of the independence of the urn probabilities, we are able to obtain closed form expressions
for the probability of finding a red ball and the expected cost of a block policy, which are stated in the
following results.

\begin{mycor} \label{cor: independence}
Given an independent multi-urn search problem, the probability distribution of $w(\vx(t),u(t))$ at any stage $t$ is 
\[
\P(w(\vx(t),u(t))=0) = \frac{N_{u(t)}-x_{u(t)}(t+1)\varphi_{u(t)}}{N_{u(t)}-x_{u(t)}(t)\varphi_{u(t)}}.
\]
\end{mycor}

\begin{mycor} \label{cor: independence_block}
Given an independent multi-urn search problem and a block policy $\mathbf{u}_{B}=(v^{1},v^{2},\ldots,v^{|\mathcal{V}|}), \ v^{i} \in \mathcal{V}$, such that $\tau(i)=\sum_{j=1}^{i-1} N_{v^{j}}$ is the first stage in which urn $v^{i}$ is queried.  Then,
\[
\prod_{t=\tau(i)}^{\tau(i)+N_{i}-1} \P(w(\vx(t),u(t))=0) = (1-\varphi_{v^{i}}),
\]
the contribution of urn $v^{i}$ to the total expected cost is 
\[
\sum_{k=\tau(i)}^{\tau(i)+N_{i}-1} \prod_{t=0}^{k} \P(w(\vx(t),u(t))=0) = \left(N_{v^{i}}-\frac{(N_{v^{i}}+1)\varphi_{v^{i}}}{2}\right)\prod_{j=1}^{i-1}(1-\varphi_{v^{j}}),
\]
and the total expected cost is
\[
\E[C] = \sum_{i=1}^{|\mathcal{V}|}
\left(N_{i}-\frac{(N_{i}+1)\varphi_{i}}{2}\right)\prod_{j=1}^{i-1}(1-\varphi_{j}).
\]
\end{mycor}
Independence implies that knowing the composition of marbles in urn $i$ does not provide any additional information on the compositions of marbles in any of the other urns.  Drawing a marble from urn $u(t)$ in stage $t$ still results in an update to this urn's probability in stage $t+1$, but all other urn probabilities remain stationary in this state transition. This property enables us to characterize the optimal policy in the case of independent urns.

\begin{mythm} \label{thm: independence_optimality}
{\bf (Independent Urns Optimality)}
Given an independent multi-urn search problem,  a block policy 
\[
\mathbf{u}_{B} = (v^{1},v^{2},\ldots,v^{|\mathcal{V}|})
\]
is optimal if and only if
\begin{equation}
N_{v^{i}}\left(\frac{2-\varphi_{v^{i}}}{\varphi_{v^{i}}}\right) \leq 
N_{v^{i+1}}\left(\frac{2-\varphi_{v^{i+1}}}{\varphi_{v^{i+1}}}\right), \quad  i=1,2,\ldots,|\mathcal{V}|. \label{eq: independence_optimality}
\end{equation}
\end{mythm}

\begin{proof}{Proof.}
See Section \ref{proof: independence_optimality}.
\end{proof}

We note that this policy is not greedy, i.e., it does not maximize the probability of finding the red marble at each stage.  Rather, the optimality condition in equation \eqref{eq: independence_optimality} balances the probability of immediately drawing a red marble with the probability of finding a red marble in successive draws from the same urn.

To gain intuition, consider a two-urn example in which each urn has the same probability of containing a red marble ($\varphi_{1}=\varphi_{2}$), but urn 1 has fewer marbles ($N_{1} < N_{2}$).  In this case the optimality condition in equation \eqref{eq: independence_optimality} would have us initially draw marbles from urn 1, which has the same probability of giving us the red marble as urn 2 but requires fewer draws.  

Alternatively, consider the case in which the two urns have the same number of marbles but $\varphi_{1}<\varphi_{2}$.  In this case, the optimal policy according to equation \eqref{eq: independence_optimality} would have us draw from urn 2 first, which is more likely than urn 1 to give us a red marble in the same number of draws.

In order to more clearly distinguish between a greedy policy and an optimal one, we provide one more example.  Consider the following independent multi-urn search problem with two urns.  Urn $1$ contains $N_{1}=1$ marble and has probability $\varphi_{1}=\frac{9}{16}$ of containing a red marble.  Urn $2$ contains $N_{2}=2$ marbles and has probability $\varphi_{2}=1$ of containing a red marble.   This problem admits two block policies: $\mathbf{u}_{B} = (2,1)$ and $\tilde{\mathbf{u}}_{B}=(1,2)$.  

Policy $\tilde{\mathbf{u}}_{B}$ is a greedy policy;  urn 1, which has the highest immediate probability of producing a red marble ($\frac{9}{16}$), is queried before urn 2.  The expected number of blue marbles drawn using policy $\tilde{\mathbf{u}}_{B}$ is 
\[
\E[\tilde{C}] = \frac{7}{16}+\left(\frac{7}{16}\right)\left(\frac{1}{2}\right) = \frac{21}{32}.
\]
Alternatively, if we follow policy $\mathbf{u}_{B}$ and draw from urn 2 first, then expected cost is
\[
\E[C]= \frac{1}{2},
\]
which is optimal.
 By accepting a slightly lower probability in the first draw, this policy guarantees that the red marble is found in at most two draws.  The optimality condition  in Theorem \ref{thm: independence_optimality}, equation \eqref{eq: independence_optimality} provides the best balance between the immediate and long-term benefits of each query.

\subsection{One Marble} \label{subsec: single_marble}

We now turn our attention to another special case of the multi-urn search problem in which we limit the total number of red marbles among all of the urns to a single marble.  In our network search scenario, this constraint would follow from assuming that the target user is connected to \emph{at most} one of the known accounts on Mary's list.  This might not be a valid assumption for Mary to make, but it might apply to other search scenarios both in and out of the network context.  

For example, suppose law enforcement investigators have evidence that a suspect made a single phone call from an unknown phone number during a certain period.  Having obtained phone records from all likely recipients, they must efficiently search for the phone call of interest within the records of these likely recipients.  

For a non-network example, suppose a hotel custodian, after servicing all of the hotel rooms, realizes he left his car keys in one of the rooms.  The hotel might consist of several wings, each with different numbers of rooms, and the custodian might feel the loss was more probable in certain wings.  The custodian wants to search the rooms efficiently for his keys, in order to find them before new customers begin to arrive.

This one-marble constraint imposes the strongest negative correlations between the urns: if the red marble is in urn $i$ then it cannot be in $j$, i.e.,
\[
P(A_{j}|A_{i}) = 0, \qquad i \neq j.
\]
Another way to characterize this constraint is to state that the events $A_{1},A_{2},\ldots,A_{|\mathcal{V}|}$ are disjoint.  We now formalize this notion in a definition.

\begin{mydef} \label{def: single_marble}
A single marble multi-urn search problem is a multi-urn search problem for which
\[
\varphi_{U} = 0 \qquad \forall \ U \subseteq \mathcal{V} \ \mathrm{such \ that} \ |U| > 1.
\]
\end{mydef}

We now analyze of the single marble search problem.  First we observe that
\[
\sum_{i \in \mathcal{V}} \varphi_{i} \leq 1.
\]
We allow for the possibility that this sum is strictly less than one, implying there is a chance that none of the urns contain the red marble.  If the sum is equal to one, then the assumption is that exactly one of the urns contains one red marble.  Theorem 5 specifies the single marble urn probabilities for an arbitrary state $\vx(t)$.

\begin{mythm} \label{thm: single_marble}
{\bf (Single Marble Urn Probabilities).}
Given a single marble multi-urn search problem and a search policy $\mathbf{u}$.  Then, the probability that a red marble is in urn $i$ given state $\vx(t)$, and given no red marble has been found in the first $t$ queries, is
\[
\P(A_{i}\mid\vx(t))
= \frac{
	\left(
		1-\frac{
			x_{i}(t)
		}{
			N_{i}
		}
	\right)
	\varphi_{i}
}
{
	1- \sum_{
		j \in \mathcal{V}
	}
	\frac{
		x_{j}(t)\varphi_{j}
	}{
		N_{j}
	}
}.
\]
The probability that the red marble is in all of the urns in any subset $U \subset \mathcal{V}$, where $|U|>1$ is
\[
\P\left(\bigcap_{i \in U} A_{i}\;\middle|\;\vx(t)\right) = 0.
\]
\end{mythm}

\begin{proof}{Proof.}
This result follows immediately from the definition of a single marble multi-urn search problem and Theorem \ref{thm: genprob}. \halmos
\end{proof}

Because the red ball can only be in one urn, all joint probabilities are zero.  This greatly
simplifies our analysis and allows us to obtain closed form expressions
for the probability of finding a red ball and the expected cost of a block policy, which are stated in the
following results.

\begin{mycor} \label{cor: single_marble}
Given a single marble multi-urn search problem and a valid search policy $\mathbf{u}$, the probability distribution of $w(\vx(t),u(t))$, conditioned on not having found a red marble in a previous stages, is
\[
\P(w(\vx(t),u(t))=k)
= \begin{cases}
	\frac{
		1-\sum_{
			j \in \mathcal{V}
		}
		\left(
			\frac{
				x_{j}(t+1)\varphi_{j}
			}
			{
				N_{j}
			}
		\right)		
	}{
		1-\sum_{
			j \in \mathcal{V}
		}
		\left(
			\frac{
				x_{j}(t)\varphi_{j}
			}
			{
				N_{j}
			}
		\right)
	},
& k = 0
\\
\frac{1}{N_{u(t)}}
\left(
	\frac{
		\varphi_{u(t)}
	}{
		1-\sum_{
			j \in \mathcal{V}
		}
		\left(
			\frac{
				x_{j}(t)\varphi_{j}
			}
			{
				N_{j}
			}
			\right)
	}
\right) ,
& k=1.
\end{cases}
\]
\end{mycor}

\begin{mycor} \label{cor: single_marble_block}
Given a single marble multi-urn problem and a block policy $\mathbf{u}_{B}=(v^{1},v^{2},\ldots,v^{|\mathcal{V}|}), \ v^{i} \in \mathcal{V}$, such that $\tau(i)=\sum_{j=1}^{i-1} N_{v^{j}}$ is the first stage in which urn $v^{i}$ is queried. Then, the probability of not finding the red marble before reaching stage $\tau(i)$ is
\[
\prod_{t=0}^{\tau(i)} \P(w(\vx(t),u(t))=0) = 
1-\sum_{j=1}^{i-1} \varphi_{v^{j}},
\]
the contribution of urn $v^{i}$ to the total expected cost is 
\[
\sum_{k=\tau(i)}^{\tau(i)+N_{i}-1} \prod_{t=0}^{k} \P(w(\vx(t),u(t))=0) = \left(
	N_{i}-\frac{(N_{i}+1)\varphi_{i}}{2}
	-N_{i}\sum_{j=1}^{i-1}\varphi_{j}
\right),
\]
and the total expected cost is
\[
\E[C] = \sum_{i = 1}^{|\mathcal{V}|}
\left(
N_{v^{i}}
-\frac{(N_{v^{i}}+1)\varphi_{v^{i}}}
{2}
\right)
-\sum_{i=1}^{|\mathcal{V}|-1} \sum_{j=i+1}^{|\mathcal{V}|}
N_{v^{j}}\varphi_{v^{i}}
.
\]
\end{mycor}

Theorem \ref{thm: single_marble_optimality} characterizes the optimal solution in the single marble case.

\begin{mythm} \label{thm: single_marble_optimality}
{\bf (Single Marble Optimality).}
Given a single marble multi-urn search problem,
a block policy 
\[
\mathbf{u}_{B} = (v^{1},v^{2},\ldots,v^{|\mathcal{V}|})
\]
is an optimal policy if and only if
\begin{equation}
\frac{\varphi_{v^{i}}}
{N_{v^{i}}}
\geq 
\frac{\varphi_{v^{i+1}}}
{N_{v^{i+1}}}, \quad  i=1,2,\ldots,|\mathcal{V}|. \label{eq: single_marble_optimality}
\end{equation}
\end{mythm}

\begin{proof}{Proof.}
See Section \ref{proof: single_marble_optimality}.
\end{proof}

The optimality condition given in equation \eqref{eq: single_marble_optimality} leads to a greedy policy in which, at each stage, the marble that is drawn is the one that is most likely to be red.  At stage $t=0$ this is certainly true, as the probability of drawing a red marble from any urn $i$ in the first draw is $\varphi_{i}/N_{i}$, which is exactly what the optimality condition optimizes.  In the next section we show that that this condition is maintained through state transitions in an optimal policy.  If drawing a marble from urn $i$  has the highest probability of producing a red marble in stage $t$, then (assuming the urn has at least one marble remaining) drawing another marble from the same urn maximizes the probability of finding a red marble in stage $t+1$.

\subsection{Monotonicity Properties}

We now provide a few monotonicity properties that give additional insight into the dynamics of multi-urn search problems, as well as block policy optimality.

\begin{mythm} \label{thm: monotonicity}
{\bf (Monotonicity).}
Given a multi-urn search problem and a search policy $\vu$, the following inequalities hold:
\begin{enumerate}
\item For any subset of urns $U\subseteq \mathcal{V}$ such that $u(t) \in U$,
\[
\P\left(\bigcap_{i \in U} A_{i}\;\middle|\;\vx(t)\right) \geq \P\left(\bigcap_{i \in U} A_{i}\;\middle|\;\vx(t+1)\right),
\]
with equality holding only in cases in which $\P\left(A_{u(t)}\mid\vx(t)\right) = 1$ or $\P\left(\bigcap_{i \in U}A_{i}\mid\vx(t)\right)=0$.
\item For any stage $t$ for which urn $u(t)$ has more than one marble remaining, i.e., $N_{u(t)}-x_{u(t)}(t)>1$,
\[
\P(w(\vx(t),u(t))=1) \leq \P (w(\vx(t+1),u(t))=1),
\]
with equality holding only when $\P(w(\vx(t),u(t))=1)=0$.
\end{enumerate}
\end{mythm}

\begin{proof}{Proof.}
See Section \ref{proof: monotonicity}.
\end{proof}

These monotonicity properties provide intuition into why optimal block policies exist.  Suppose at stage $t$ a blue marble is drawn from urn $u(t)$.  At stage $t+1$, the probability of urn $u(t)$ containing a red marble has decreased as a result of this new information.  However, the probability that the next marble drawn from urn $u(t)$ is red has \emph{increased} from the previous stage.  If drawing from urn $u(t)$ in stage $t$ had a high probability of returning a red marble, drawing another marble from $u(t)$ in stage $t+1$ has an even higher probability of producing a red marble.

In the case of independent urns, this property provides justification for using a block policy.  Suppose urn $u(t)$ has multiple marbles in it and is optimal at stage $t$, and a blue marble is drawn from this urn.  At stage $t+1$ the probability of drawing a red marble from urn $u(t)$ has increased, while all other urn and marble probabilities have remained unchanged from the previous stage $t$.  It follows that it would continue to be optimal to draw from urn $u(t)$.

If we allow for correlations among the urns, however, drawing a blue marble from urn $u(t)$ might also increase the probability of drawing a red marble from other urns in the following stage.  In the single marble multi-urn search problem, drawing a blue marble from urn $u(t)$ in stage $t$ increases the probability of finding a red marble in each of the other urns in stage $t+1$.  We now provide our final theoretical result, which states that the rate of probability growth in a queried urn is always at least as large as the rate of probability growth in any other urn.

\begin{mythm} \label{thm: probrate}
{\bf (Marble Probability Bound).}
Given a multi-urn search problem and a search policy $\vu$, the following inequality holds:
\[
\frac{\P\left(w(\vx(t+1),u(t))=1\right)}{\P\left(w(\vx(t),u(t))=1\right)}
\geq
\frac{\P\left(w(\vx(t+1),i)=1\right)}{\P\left(w(\vx(t),i)=1\right)}.
\]
\end{mythm}

\begin{proof}{Proof.}
See Section \ref{proof: probrate}.
\end{proof}


Theorem \ref{thm: probrate} provides much intuition about why optimal block policies always exist in multi-urn search problems. It also shows that a purely greedy strategy, in which the probability of immediately drawing a red marble is maximized at each stage, will always produce a (possibly suboptimal) block policy. 
Finally, it provides some insight into why the optimality conditions for the single marble multi-urn search problem given in equation \eqref{eq: single_marble_optimality} result in a greedy policy.  Equation \eqref{eq: single_marble_optimality} specifies that the first marble drawn is the one that maximizes the probability of immediately finding the target.  It follows from Theorem \ref{thm: probrate} that subsequent draws from the same urn will continue to maximize this probability.

\subsection{Urn Correlation Dynamics}

In the preceding section we examined how urn probabilities and the probability of drawing a red marble from each urn evolved as a function of the state of the system.  In this section we show by example how, in general, the correlations among the urns can evolve in ways that we find to be counterintuitive. 

We say that two urns $i$ and $j$ are \emph{positively correlated} at stage $t$ if they have positive covariance, i.e.,
\[
\P\left(A_{i}\cap A_{j}\mid\vx(t)\right) > \P\left(A_{i}\mid\vx(t)\right)\P\left(A_{j}\mid\vx(t)\right).
\]
Likewise, urns $i$ and $j$ are \emph{negatively correlated} if their covariance is negative,
\[
\P\left(A_{i}\cap A_{j}\mid\vx(t)\right) < \P\left(A_{i}\mid\vx(t)\right)\P\left(A_{j}\mid\vx(t)\right).
\]
In Theorem \ref{thm: independence} we showed the somewhat intuitive result that independence among the urn probabilities is preserved through state transitions.  
In general, correlations can change through Bayesian updates each time a blue marble is drawn.  These changes can include changes in sign.  

We now provide an example scenario in which all correlations are positive in the initial conditions, but after a blue marble is drawn some correlations become negative. 
Suppose we have a multi-urn search problem consisting of three urns.  Each urn $i$ has $N_{i}=1$ marble and initial probability $\varphi_{i}=\frac{1}{2}$ of containing a red marble.  Furthermore,
\begin{align*}
\varphi_{\{1,2\}} = \varphi_{\{1,3\}} = \varphi_{\{2,3\}} = \varphi_{\{1,2,3\}} = \frac{1}{3}.
\end{align*}
A quick calculation confirms that this is a valid probability model, and that a red marble exists in at least one of the three urns with probability $\frac{5}{6}$.  We also verify that all correlations are positive,
\begin{align*}
\varphi_{\{1,2\}} &> \varphi_{1}\varphi_{2} \\
\varphi_{\{1,3\}} &> \varphi_{1}\varphi_{3} \\
\varphi_{\{2,3\}} &> \varphi_{2}\varphi_{3} .
\end{align*}
\begin{figure}[!hbt]
\centering
\includegraphics[width=0.4\textwidth]{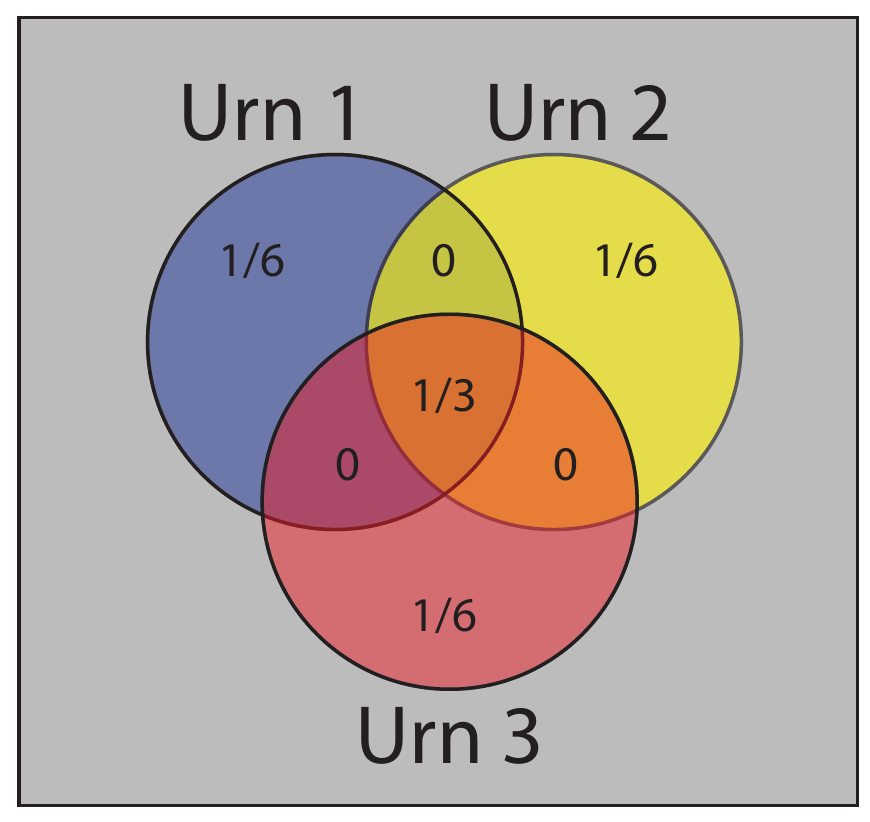}
\caption{Venn diagram of the probabilities in the three-urn example.}\label{fig: venn}
\end{figure}
Also, we have a more general positive correlation property,
\[
\varphi_{\{1,2,3\}} > \varphi_{1}\varphi_{2}\varphi_{3}.
\]
Figure \ref{fig: venn} depicts this probability law in a Venn diagram.  Note that there is no probability of exactly two urns containing red marbles.  A red marble is present in zero, one, or three urns almost surely.

Now suppose at stage $t=0$, a blue marble is drawn from urn $1$.  From Theorem \ref{thm: genprob}, the stage 1 urn probabilities are
\begin{align*}
\P\left(A_{1}\mid\vx(1)\right) & = 0 \\
\P\left(A_{2}\mid\vx(1)\right) & = \frac{1}{3}\\
\P\left(A_{3}\mid\vx(1)\right) & = \frac{1}{3}\\
\P\left(A_{1}\cap A_{2}\mid\vx(1)\right) & = 0\\
\P\left(A_{1}\cap A_{3} \mid\vx(1)\right) & = 0\\
\P\left(A_{2}\cap A_{3} \mid\vx(1)\right) & = 0\\
\P\left(A_{1} \cap A_{2} \cap A_{3} \mid\vx(1)\right) & = 0.
\end{align*}

By eliminating the possibility that each urn contained a red marble, the only outcomes that have positive probability in stage 1 are single-marble outcomes.  The correlation between urns 2 and 3, which was positive in the initial conditions, has become negative in stage 1:
\[
\P\left(A_{2}\cap A_{3}\mid\vx(1)\right) = 0 < \frac{1}{9} = \P\left(A_{2}\mid\vx(1)\right)\P\left(A_{3}\mid\vx(1)\right).
\]

One could similarly produce examples in which correlations that were originally negative become positive after drawing blue marbles.  In the single-marble and independent urn cases, for which we provided characterizations of the optimal policies in the preceding sections, correlations among the urns exhibited some stationarity with respect to stage.  In general, the nature of correlations among the urns can change substantially as blue marbles are drawn and probabilities are updated.  This characteristic presents a challenge to finding characterizations of the optimal policy in general.

\section{Proofs of Theorems} \label{proofs}

In this section we provide the technical proof for each theorem.

\subsection{Theorem \ref{thm: genprob}}
\label{proof: genprob}
\begin{proof}{Proof of Theorem \ref{thm: genprob}}
\begin{proof}{}
Substituting the initial condition, $\vx(0)=\mathbf{0}$, into the result returns the prior
\[
\P\left(\bigcap_{i\in U} A_{i}\;\middle|\;\vx(0)\right) = \varphi_{U} := \P\left(\bigcap_{i\in U} A_{i}\right).
\]

The proof proceeds by induction.  First note that in order to reach state $\vx(t+1)$ from stage $t$, a blue marble must have been drawn from urn $u(t)$ from state $\vx(t)$.  We use the law of total probability to decompose this event and form a recursion:
\begin{align}
&\P\left(\bigcap_{i\in U} A_{i}\;\middle|\;\vx(t+1)\right)  = \frac{\P\left(\left(\bigcap_{i\in U} A_{i}\right) \cap w(\vx(t),u(t))=0|\vx(t)\right)}{\P(w(\vx(t),u(t))=0)} \nonumber \\
& = \frac{\left(1-\frac{1}{N_{u(t)}-x_{u(t)}(t)}\right)\P\left(\bigcap_{i\in U\cup u(t)} A_{i}\;\middle|\;\vx(t)\right) + \P(A_{u(t)}^{c}\cap\left(\bigcap_{i\in U} A_{i}\right)|\vx(t))}{1-\left(\frac{1}{N_{u(t)}-x_{u(t)}(t)}\right)\P(A_{u(t)}|\vx(t))} \nonumber \\
& = \frac{\P\left(\bigcap_{i\in U} A_{i}\;\middle|\;\vx(t)\right)
-\left(\frac{1}{N_{u(t)}-x_{u(t)}(t)}\right)\P\left(\bigcap_{i\in U\cup u(t)}A_{i}\;\middle|\;\vx(t)\right)}
{1-\left(\frac{1}{N_{u(t)}-x_{u(t)}(t)}\right)\P(A_{u(t)}|\vx(t))}
\label{eq: probrecursion}
\end{align}

The result in Theorem \ref{thm: genprob} forms our induction hypothesis.  We use this result to form the three probabilities in the recursion given in equation \eqref{eq: probrecursion}.
%
%
\begin{align*}
&\P\left(\bigcap_{i\in U} A_{i}\;\middle|\;\vx(t)\right) = 
\frac{
	\left[
		\prod_{i \in U}
		\left(
			1-\frac{
					x_{i}(t)
				}{
					N_{i}
				}
		\right)
	\right]
	\sum_{
		\{S \subseteq \mathcal{V}: \; S \supseteq U\}
	}
	(-1)^{|S|-|U|}\varphi_{S}
	\prod_{
		j \in S \setminus U
	}
	\frac{
		x_{j}(t)
	}{
		N_{j}
	}
}{
	\sum_{S \subseteq \mathcal{V}}
	(-1)^{|S|}
	\varphi_{S} 
	\prod_{j \in S} 
	\frac{
		x_{j}(t)
	}{
		N_{j}
	}
}
\\
& \P\left(A_{u(t)}|\vx(t)\right)   = 
\frac{
	\left(
			1-\frac{
					x_{u(t)}(t)
				}{
					N_{u(t)}
				}
	\right)
	\sum_{
		\{S \subseteq \mathcal{V}: \; u(t) \in S\}
	}
	(-1)^{|S|-1}\varphi_{S}
	\prod_{
		j \in S \setminus \{u(t)\}
	}
	\frac{
		x_{j}(t)
	}{
		N_{j}
	}
}{
	\sum_{S \subseteq \mathcal{V}}
	(-1)^{|S|}
	\varphi_{S} 
	\prod_{j \in S} 
	\frac{
		x_{j}(t)
	}{
		N_{j}
	}
}
\\ 
 &
	    \P\left(\bigcap_{i\in U\cup u(t)}A_{i}\;\middle|\;\vx(t)\right) \\
& \qquad = \begin{cases}
	    \frac{
	\left[
		\prod_{i \in U}
		\left(
			1-\frac{
					x_{i}(t)
				}{
					N_{i}
				}
		\right)
	\right]
	\sum_{
		\{S \subseteq \mathcal{V}: \; S \supseteq U\}
	}
	(-1)^{|S|-|U|}\varphi_{S}
	\prod_{
		j \in S \setminus U
	}
	\frac{
		x_{j}(t)
	}{
		N_{j}
	}
}{
	\sum_{S \subseteq \mathcal{V}}
	(-1)^{|S|}
	\varphi_{S} 
	\prod_{j \in S} 
	\frac{
		x_{j}(t)
	}{
		N_{j}
	}
}
& u(t) \in U \\
	    \frac{
	\left[
		\prod_{i \in (U \cup u(t))}
		\left(
			1-\frac{
					x_{i}(t)
				}{
					N_{i}
				}
		\right)
	\right]
	\sum_{
		\{S \subseteq \mathcal{V}: \; S \supseteq (U \cup u(t))\}
	}
	(-1)^{|S|-|U|-1}\varphi_{S}
	\prod_{
		j \in S \setminus (U \cup u(t))
	}
	\frac{
		x_{j}(t)
	}{
		N_{j}
	}
}{
	\sum_{S \subseteq \mathcal{V}}
	(-1)^{|S|}
	\varphi_{S} 
	\prod_{j \in S} 
	\frac{
		x_{j}(t)
	}{
		N_{j}
	}
}
& u(t) \notin U. 
\end{cases}
\end{align*}	 

As we see from these probabilities, we have two cases to consider:
	\begin{enumerate}
		\item $u(t) \in U$.
		\item $u(t) \notin U$.
	\end{enumerate}
We now look at each of these cases individually.

{\bf Case 1:} $u(t) \in U$.
We begin by substituting the probabilities formed using the induction hypothesis into the recursion in equation \eqref{eq: probrecursion}.
\begin{align*}
& \P\left(\bigcap_{i\in U} A_{i}\;\middle|\;\vx(t+1)\right)
=
\frac{\P\left(\bigcap_{i\in U} A_{i}\;\middle|\;\vx(t)\right)
-\left(\frac{1}{N_{u(t)}-x_{u(t)}(t)}\right)\P\left(\bigcap_{i\in U\cup u(t)}A_{i}\;\middle|\;\vx(t)\right)}
{1-\left(\frac{1}{N_{u(t)}-x_{u(t)}(t)}\right)\P(A_{u(t)}|\vx(t))}
\\
& \quad = 
\frac{
	\left(1-\frac{1}{N_{u(t)}-x_{u(t)}(t)}\right)
	\frac{
	\left[
		\prod_{i \in U}
		\left(
			1-\frac{
					x_{i}(t)
				}{
					N_{i}
				}
		\right)
	\right]
	\sum_{
		\{S \subseteq \mathcal{V}: \; S \supseteq U\}
	}
	(-1)^{|S|-|U|}\varphi_{S}
	\prod_{
		j \in S \setminus U
	}
	\frac{
		x_{j}(t)
	}{
		N_{j}
	}
}{
	\sum_{S \subseteq \mathcal{V}}
	(-1)^{|S|}
	\varphi_{S} 
	\prod_{j \in S} 
	\frac{
		x_{j}(t)
	}{
		N_{j}
	}
}
}
{
	1-\left(
		\frac{1}{
			N_{u(t)}-x_{u(t)}(t)
		}
	\right)
	\left(
		\frac{
			\left(
				1-\frac{
					x_{u(t)}(t)
				}{
					N_{u(t)}
				}
			\right)
			\sum_{
				\{S \subseteq \mathcal{V}: \; u(t) \in S\}
			}
			(-1)^{|S|-1}\varphi_{S}
			\prod_{
				j \in S \setminus \{u(t)\}
			}
			\frac{
				x_{j}(t)
			}{
				N_{j}
			}
		}{
			\sum_{S \subseteq \mathcal{V}}
			(-1)^{|S|}
			\varphi_{S} 
			\prod_{j \in S} 
			\frac{
				x_{j}(t)
			}{
				N_{j}
			}
		}
	\right)
}
\\
&\quad = 
\frac{
	\left(\frac{N_{u(t)}-x_{u(t)}(t)-1}{N_{u(t)}-x_{u(t)}(t)}\right)
	\frac{
	\left[
		\left(
			\frac{N_{u(t)}-x_{u(t)}(t)}{N_{u(t)}}
		\right)
		\prod_{
			i \in U \setminus \{u(t)\}
		}
		\left(
			1-\frac{
					x_{i}(t)
				}{
					N_{i}
				}
		\right)
	\right]
	\sum_{
		\{S \subseteq \mathcal{V}: \; S \supseteq U\}
	}
	(-1)^{|S|-|U|}\varphi_{S}
	\prod_{
		j \in S \setminus U
	}
	\frac{
		x_{j}(t)
	}{
		N_{j}
	}
}{
	\sum_{S \subseteq \mathcal{V}}
	(-1)^{|S|}
	\varphi_{S} 
	\prod_{j \in S} 
	\frac{
		x_{j}(t)
	}{
		N_{j}
	}
}
}
{
	1-\left(
		\frac{1}{
			N_{u(t)}
		}
	\right)
	\left(
		\frac{
			\sum_{
				\{S \subseteq \mathcal{V}: \; u(t) \in S\}
			}
			(-1)^{|S|-1}\varphi_{S}
			\prod_{
				j \in S \setminus \{u(t)\}
			}
			\frac{
				x_{j}(t)
			}{
				N_{j}
			}
		}{
			\sum_{S \subseteq \mathcal{V}}
			(-1)^{|S|}
			\varphi_{S} 
			\prod_{j \in S} 
			\frac{
				x_{j}(t)
			}{
				N_{j}
			}
		}
	\right)
}
\end{align*}

We proceed by separating the summations into terms corresponding to sets containing $u(t)$ and those that do not contain $u(t)$.  We can then factor out the terms corresponding to urn $u(t)$ and make the following substitutions:
	\begin{equation*}
	    x_{i}(t) = \begin{cases}
	    x_{i}(t+1) & i \neq u(t) \\
	    x_{i}(t+1) - 1 & i = u(t).
	    \end{cases}
	\end{equation*}
Continuing the simplification from above,
\begin{align*}
& = 
\frac{
	\left(\frac{N_{u(t)}-x_{u(t)}(t+1)}{N_{u(t)}}\right)
	\left[
		\prod_{
			i \in U \setminus \{u(t)\}
		}
		\left(
			1-\frac{
					x_{i}(t+1)
				}{
					N_{i}
				}
		\right)
	\right]
	\sum_{
		\{S \subseteq \mathcal{V}: \; S \supseteq U\}
	}
	(-1)^{|S|-|U|}\varphi_{S}
	\prod_{
		j \in S \setminus U
	}
	\frac{
		x_{j}(t+1)
	}{
		N_{j}
	}
}
{
	\sum_{S \subseteq \mathcal{V}}
	(-1)^{|S|}
	\varphi_{S} 
	\prod_{j \in S} 
	\frac{
		x_{j}(t)
	}{
		N_{j}
	}
	-
	\left(
		\frac{1}{
			N_{u(t)}
		}
	\right)
	\left(
		\sum_{
			\{S \subseteq \mathcal{V}: \; u(t) \in S\}
		}
		(-1)^{|S|-1}\varphi_{S}
		\prod_{
			j \in S \setminus \{u(t)\}
		}
		\frac{
			x_{j}(t)
		}{
			N_{j}
		}
	\right)
}
\\
& = 
\frac{
	\left[
		\prod_{
			i \in U
		}
		\left(
			1-\frac{
					x_{i}(t+1)
				}{
					N_{i}
				}
		\right)
	\right]
	\sum_{
		\{S \subseteq \mathcal{V}: \; S \supseteq U\}
	}
	(-1)^{|S|-|U|}\varphi_{S}
	\prod_{
		j \in S \setminus U
	}
	\frac{
		x_{j}(t+1)
	}{
		N_{j}
	}
}
{
	\sum_{S \subseteq \mathcal{V} \setminus u(t)}
	(-1)^{|S|}
	\varphi_{S} 
	\prod_{j \in S} 
	\frac{
		x_{j}(t)
	}{
		N_{j}
	}
	+
	\left(
		\frac{
			x_{u(t)}(t)
		}{
			N_{u(t)}
		}
		+ \frac{1}{
			N_{u(t)}
		}
	\right)
	\left(
		\sum_{
			\{S \subseteq \mathcal{V}: \; u(t) \in S\}
		}
		(-1)^{|S|}\varphi_{S}
		\prod_{
			j \in S \setminus \{u(t)\}
		}
		\frac{
			x_{j}(t)
		}{
			N_{j}
		}
	\right)
}
\\
& = 
\frac{
	\left[
		\prod_{
			i \in U
		}
		\left(
			1-\frac{
					x_{i}(t+1)
				}{
					N_{i}
				}
		\right)
	\right]
	\sum_{
		\{S \subseteq \mathcal{V}: \; S \supseteq U\}
	}
	(-1)^{|S|-|U|}\varphi_{S}
	\prod_{
		j \in S \setminus U
	}
	\frac{
		x_{j}(t+1)
	}{
		N_{j}
	}
}
{
	\sum_{S \subseteq \mathcal{V} \setminus u(t)}
	(-1)^{|S|}
	\varphi_{S} 
	\prod_{j \in S} 
	\frac{
		x_{j}(t+1)
	}{
		N_{j}
	}
	+
		\sum_{
			\{S \subseteq \mathcal{V}: \; u(t) \in S\}
		}
		(-1)^{|S|}\varphi_{S}
		\prod_{
			j \in S 
		}
		\frac{
			x_{j}(t+1)
		}{
			N_{j}
		}
}
\\
& = 
\frac{
	\left[
		\prod_{
			i \in U
		}
		\left(
			1-\frac{
					x_{i}(t+1)
				}{
					N_{i}
				}
		\right)
	\right]
	\sum_{
		\{S \subseteq \mathcal{V}: \; S \supseteq U\}
	}
	(-1)^{|S|-|U|}\varphi_{S}
	\prod_{
		j \in S \setminus U
	}
	\frac{
		x_{j}(t+1)
	}{
		N_{j}
	}
}
{
	\sum_{S \subseteq \mathcal{V}}
	(-1)^{|S|}
	\varphi_{S} 
	\prod_{j \in S} 
	\frac{
		x_{j}(t+1)
	}{
		N_{j}
	}
}.
\end{align*}
Observe that this is the desired result for stage $t+1$.  We now provide the induction step for the case in which $u(t) \notin U$.

{\bf Case 2:} $u(t) \notin U$ follows a similar set of steps.  We begin by substituting the probabilities formed using the induction hypothesis into the recursion in equation \eqref{eq: probrecursion}.
\begin{align*}
& \frac{\P\left(\bigcap_{i\in U} A_{i}\;\middle|\;\vx(t)\right)
-\left(\frac{1}{N_{u(t)}-x_{u(t)}(t)}\right)\P\left(\bigcap_{i\in U\cup u(t)}A_{i}\;\middle|\;\vx(t)\right)}
{1-\left(\frac{1}{N_{u(t)}-x_{u(t)}(t)}\right)\P(A_{u(t)}|\vx(t))}
\\
& \ = 
\Bigg[
\frac{
	\left(
		\prod_{i \in U}
		\left(
			1-\frac{
				x_{i}(t)
			}{
				N_{i}
			}
		\right)
	\right)
	\sum_{
		\{S \subseteq \mathcal{V}: \; S \supseteq U\}
	}
	(-1)^{|S|-|U|}\varphi_{S}
	\prod_{
		j \in S \setminus U
	}
	\frac{
		x_{j}(t)
	}{
		N_{j}
	}
}{
		\sum_{S \subseteq \mathcal{V}}
		(-1)^{|S|}
		\varphi_{S} 
		\prod_{j \in S} 
		\frac{
			x_{j}(t)
		}{
			N_{j}
		}
}
\\
& \quad 
-
\frac{
    \left(
    	\frac{1}{N_{u(t)}-x_{u(t)}(t)}
    \right)
    	\left(
    		\prod_{i \in U \cup u(t)}
    		\left(
    			1-\frac{
    				x_{i}(t)
    			}{
    				N_{i}
    			}
    		\right)
    	\right)
    	\sum_{
    		\{S \subseteq \mathcal{V}: \; S \supseteq U \cup u(t)\}
    	}
    	(-1)^{|S|-|U|-1}\varphi_{S}
    	\prod_{
    		j \in S \setminus (U \cup u(t))
    	}
    	\frac{
    		x_{j}(t)
    	}{
    		N_{j}
    	}
}
{
		\sum_{S \subseteq \mathcal{V}}
		(-1)^{|S|}
		\varphi_{S} 
		\prod_{j \in S} 
		\frac{
			x_{j}(t)
		}{
			N_{j}
		}
}
\Bigg]
\\
& \qquad \times
\Bigg[
	1-\left(
		\frac{1}{
			N_{u(t)}-x_{u(t)}(t)
		}
	\right)
	\left(
		\frac{
			\left(
				1-\frac{
					x_{u(t)}(t)
				}{
					N_{u(t)}
				}
			\right)
			\sum_{
				\{S \subseteq \mathcal{V}: \; u(t) \in S\}
			}
			(-1)^{|S|-1}\varphi_{S}
			\prod_{
				j \in S \setminus \{u(t)\}
			}
			\frac{
				x_{j}(t)
			}{
				N_{j}
			}
		}{
			\sum_{S \subseteq \mathcal{V}}
			(-1)^{|S|}
			\varphi_{S} 
			\prod_{j \in S} 
			\frac{
				x_{j}(t)
			}{
				N_{j}
			}
		}
	\right)
\Bigg]^{-1}
\end{align*}

The denominators in the above expression reduce in exactly the same way as in the previous case in which $u(t) \in U$.  In fact, the denominators in the induction hypothesis and in equation \eqref{eq: probrecursion} do not depend on whether $u(t) \in U$.  Because we have already shown the steps for reducing this denominator to the desired form, we omit these steps and only show the induction on the numerators for this case:

\begin{align*}
& 
\Bigg[
	\left(
		\prod_{i \in U}
		\left(
			1-\frac{
				x_{i}(t)
			}{
				N_{i}
			}
		\right)
	\right)
	\sum_{
		\{S \subseteq \mathcal{V}: \; S \supseteq U\}
	}
	(-1)^{|S|-|U|}\varphi_{S}
	\prod_{
		j \in S \setminus U
	}
	\frac{
		x_{j}(t)
	}{
		N_{j}
	}
\\
& \qquad -\left(\frac{1}{N_{u(t)}-x_{u(t)}(t)}\right)
	\left(
		\prod_{i \in U \cup u(t)}
		\left(
			1-\frac{
				x_{i}(t)
			}{
				N_{i}
			}
		\right)
	\right)
	\sum_{
		\{S \subseteq \mathcal{V}: \; S \supseteq U \cup u(t)\}
	}
	(-1)^{|S|-|U|-1}\varphi_{S}
	\prod_{
		j \in S \setminus (U \cup u(t))
	}
	\frac{
		x_{j}(t)
	}{
		N_{j}
	}
\Bigg]
\\
&= 
\left(
	\prod_{i \in U}
	\left(
		1-\frac{
			x_{i}(t)
		}{
			N_{i}
		}
	\right)
\right)
\Bigg[
	\sum_{
		\{S \subseteq \mathcal{V}: \; S \supseteq U\}
	}
	(-1)^{|S|-|U|}\varphi_{S}
	\prod_{
		j \in S \setminus U
	}
	\frac{
		x_{j}(t)
	}{
		N_{j}
	}
\\
& \qquad +\left(\frac{1}{N_{u(t)}-x_{u(t)}(t)}\right)
	\left(
	    \frac{N_{u(t)}-x_{u(t)}(t)}{N_{u(t)}}
	\right)
	\sum_{
		\{S \subseteq \mathcal{V}: \; S \supseteq U \cup u(t)\}
	}
	(-1)^{|S|-|U|}\varphi_{S}
	\prod_{
		j \in S \setminus (U \cup u(t))
	}
	\frac{
		x_{j}(t)
	}{
		N_{j}
	}
\Bigg]
\\
&= 
\left(
	\prod_{i \in U}
	\left(
		1-\frac{
			x_{i}(t)
		}{
			N_{i}
		}
	\right)
\right)
\Bigg[
	\sum_{
		\{S \subseteq \mathcal{V} \setminus \{u(t)\}: \; S \supseteq U\}
	}
	(-1)^{|S|-|U|}\varphi_{S}
	\prod_{
		j \in S \setminus U
	}
	\frac{
		x_{j}(t)
	}{
		N_{j}
	}
\\
& \qquad 
	+\left(
		\frac{
			x_{u(t)}(t)
		}{
			N_{u(t)}
		}
		+
		\frac{1}{N_{u(t)}}
	\right)
	\sum_{
		\{S \subseteq \mathcal{V}: \; S \supseteq U \cup u(t)\}
	}
	(-1)^{|S|-|U|}\varphi_{S}
	\prod_{
		j \in S \setminus (U \cup u(t))
	}
	\frac{
		x_{j}(t)
	}{
		N_{j}
	}
\Bigg]
\end{align*}
We again make the substitution:
	\begin{equation*}
	    x_{i}(t) = \begin{cases}
	    x_{i}(t+1) & i \neq u(t) \\
	    x_{i}(t+1) - 1 & i = u(t),
	    \end{cases}
	\end{equation*}
and continue from above:
\begin{align*}
&= 
\left(
	\prod_{i \in U}
	\left(
		1-\frac{
			x_{i}(t+1)
		}{
			N_{i}
		}
	\right)
\right)
\Bigg[
	\sum_{
		\{S \subseteq \mathcal{V} \setminus \{u(t)\}: \; S \supseteq U\}
	}
	(-1)^{|S|-|U|}\varphi_{S}
	\prod_{
		j \in S \setminus U
	}
	\frac{
		x_{j}(t+1)
	}{
		N_{j}
	}
&& \qquad \qquad \qquad
\\
& \qquad 
	+\left(
		\frac{
			x_{u(t)}(t+1)
		}{
			N_{u(t)}
		}
	\right)
	\sum_{
		\{S \subseteq \mathcal{V}: \; S \supseteq U \cup u(t)\}
	}
	(-1)^{|S|-|U|}\varphi_{S}
	\prod_{
		j \in S \setminus (U \cup u(t))
	}
	\frac{
		x_{j}(t+1)
	}{
		N_{j}
	}
\Bigg]
&&  \qquad\qquad \qquad
\\
& = 
\left(
	\prod_{
		i \in U
	}
	\left(
		1-\frac{
				x_{i}(t+1)
			}{
				N_{i}
			}
	\right)
\right)
\sum_{
	\{S \subseteq \mathcal{V}: \; S \supseteq U\}
}
(-1)^{|S|-|U|}\varphi_{S}
\prod_{
	j \in S \setminus U
}
\frac{
	x_{j}(t+1)
}{
	N_{j}
}
.
&&  \qquad\qquad \qquad
\end{align*}
This final expression is the numerator in Theorem \ref{thm: genprob} for the stage $t+1$ urn probabilities.  \halmos
\end{proof}

\end{proof}

\subsection{Theorem \ref{thm: blockpolicy}}
\label{proof: blockpolicy}

Before we provide a proof for Theorem \ref{thm: blockpolicy}, we state and prove the following Lemma.

\begin{mylem} \label{lem: problemma}
Given a multi-urn search problem on a set of urns $\mathcal{V}$, suppose a policy $\vu$ is executed to stage $t$ irrespective of whether a red marble is found at any stage.  Let $B_{i}$ be the event that a red marble has been drawn from urn $i \in \mathcal{V}$ in this experiment.  Then, for any subset $U \subseteq \mathcal{V}$, the probability of having drawn a red marble from each of the urns in $U$ and none of the other urns is 
\[
\P\left(\left(\bigcap_{i \in U}B_{i}\right)\cap\left(\bigcap_{j \in \mathcal{V}\setminus U}B_{j}^{c}\right)\right) = \sum_{S: S \subseteq \mathcal{V}, S \supseteq U} (-1)^{|S|-|U|}\varphi_{S}\prod_{i \in S}\frac{x_{i}(t)}{N_{i}} \geq 0.
\]
\end{mylem}

\begin{proof}{Proof of Lemma \ref{lem: problemma}.}
\begin{proof}{}
This result is comes from the principle of inclusion-exclusion, and follows from Equation \ref{eq: inclusion-exclusion}.  
From basic set operations and DeMorgan's Law, we can write
\[
\left(\bigcap_{i \in U}B_{i}\right) = \left(\left(\bigcap_{i \in U}B_{i}\right)\cap\left(\bigcap_{j \in \mathcal{V}\setminus U}B_{j}^{c}\right)\right) \cup \left(\left(\bigcap_{i \in U}B_{i}\right) \cap \left(\bigcup_{j \in \mathcal{V}\setminus U}B_{j}\right)\right),
\]
which is a union of disjoint sets.  Therefore,
\begin{equation}
\P\left(\left(\bigcap_{i \in U}B_{i}\right) \cap \left(\bigcap_{j \in \mathcal{V}\setminus U}B_{j}^{c}\right)\right) = 
\P\left(\bigcap_{i \in U}B_{i}\right)
-\P\left(\left(\bigcap_{i \in U}B_{i}\right) \cap \left(\bigcup_{j \in \mathcal{V}\setminus U}B_{j}\right)\right) \label{eq: prob_axiom3}.
\end{equation}

The probability a red marble is drawn from all of the urns in $U$ in this experiment can be found using the multiplication rule:
\begin{equation}
\P\left(\bigcap_{i \in U}B_{i}\right) = \varphi_{U}\prod_{i \in U}\frac{x_{i}(t)}{N_{i}} \label{eq: eventprob}.
\end{equation}
Recall that $\varphi_{U}$ is the probability of all of the urns in set $U$ containing a red marble, and $\frac{x_{i}(t)}{N_{i}}$ is simply the fraction of marbles removed from urn $i$ at stage $t$.  The product in this expression implies conditional independence: given all of the urns in $U$ contain a red marble, the probability that a red marble is drawn from each of them by stage $t$ is the product of the individual probabilities.  This conditional independence follows implicitly from our search assumptions.  The order of marble draws from each urn is random, and does not depend on the order of marble draws from any other urn.

We now examine $\P\left(\left(\bigcap_{i \in U}B_{i}\right) \cap \left(\bigcup_{j \in \mathcal{V}\setminus U}B_{j}\right)\right)$.  We first note that
\[
\left(\left(\bigcap_{i \in U}B_{i}\right) \cap \left(\bigcup_{j \in \mathcal{V}\setminus U}B_{j}\right)\right)
= \bigcup_{j \in \mathcal{V}\setminus U} \left(\left(\bigcap_{i \in U}B_{i}\right) \cap B_{j}\right).
\]

Using the principle of inclusion exclusion, we can find the probability of this union,
\begin{align}
& \P\left(
	\bigcup_{j \in \mathcal{V}\setminus U} \left(\left(\bigcap_{i \in U}B_{i}\right) \cap B_{j}\right)
\right) 
= \sum_{j \in \mathcal{V} \setminus U} 
\P\left(
	\left(\bigcap_{i \in U}B_{i}\right) \cap B_{j} 
\right) 
\nonumber
\\ 
& \qquad
-
\sum_{j \in \mathcal{V} \setminus U} 
\left[
    \sum_{k \in \mathcal{V} \setminus (U \cup \{j\})}
    \P\left(
    	\left(\bigcap_{i \in U}B_{i}\right) \cap B_{j} \cap B_{k}
    \right) 
\right]
\nonumber
\\
 & \qquad
 \cdots + (-1)^{|\mathcal{V}|-|U|+1}
\P\left(
	\bigcap_{i \in \mathcal{V}}B_{i}
\right).
\label{eq: unionprob}
\end{align}
Substituting expressions \ref{eq: eventprob} and \ref{eq: unionprob} into equation \ref{eq: prob_axiom3} reduces to the desired result.  The principle of inclusion-exclusion and the axioms of probability ensure that this quantity is nonnegative.
\halmos
\end{proof}

\end{proof}

We now provide the proof of Theorem \ref{thm: blockpolicy}.
\begin{proof}{Proof of Theorem \ref{thm: blockpolicy}.}

\begin{proof}{}
Suppose we are given a multi-urn search problem on a set of urns $\mathcal{V}$, with each urn $i \in \mathcal{V}$ containing $N_{i}$ marbles and initial target probabilities $\varphi_{U}$ for all $U \subseteq \mathcal{V}$.  Suppose also that valid policy $\vu=(u(0),\ldots,u(N-1))$ is optimal, where $\vu$ is not a block policy.  This implies that we can find a stage $\tau$ where
\begin{align*}
u(\tau) & = i \\
u(\tau+1),u(\tau+2),\ldots,u(\tau+\delta-1) & \neq i \\
u(\tau+\delta)&=i,
\end{align*}
for some urn $i \in \mathcal{V}$, where $\delta>1$.

\begin{figure}[!hbt]
\centering
\includegraphics[width=0.4\textwidth]{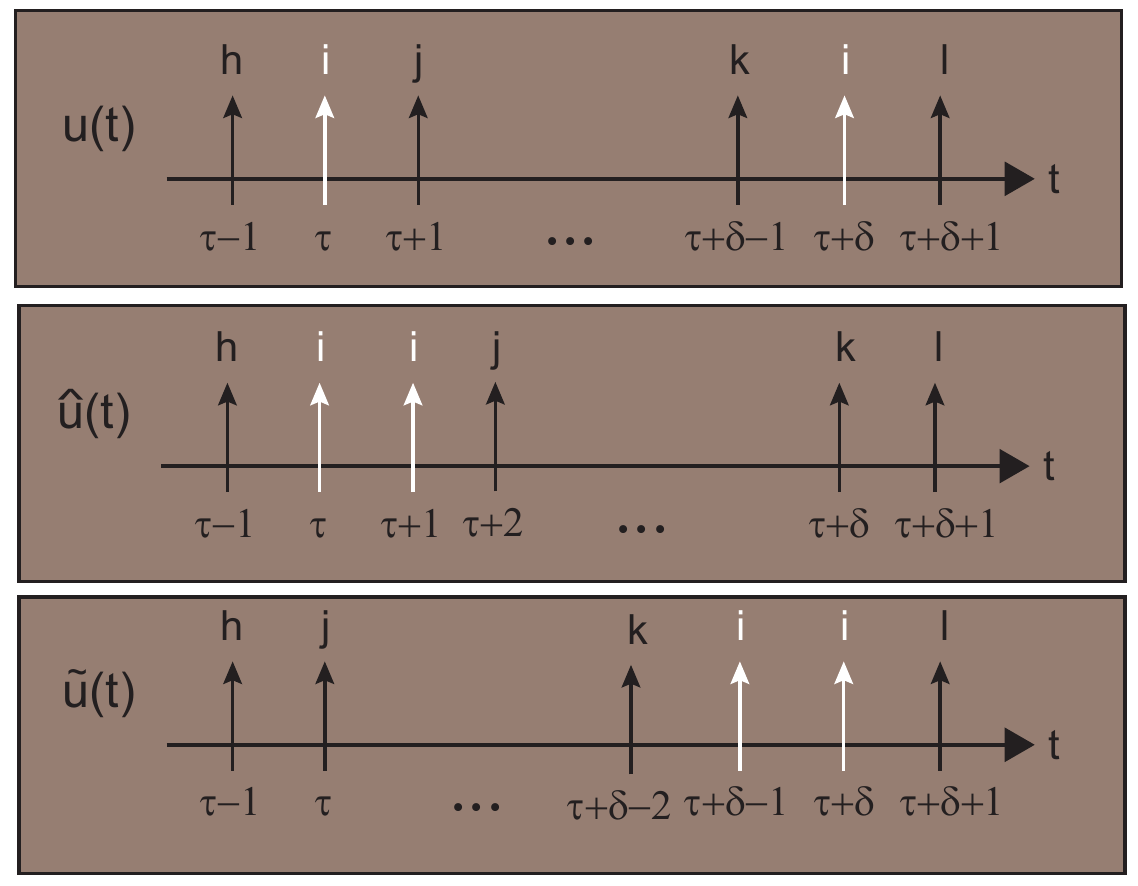}
\caption{Comparison of policies $\vu$, $\hat{\vu}$, and $\tilde{\vu}$.} 
\label{fig: policy_comparison}
\end{figure}

We now consider two alternative policies that move the queries of urn $i$ into ``blocks''.  Policy $\hat{\vu}$ executes the two queries of $i$ in stages $\tau$ and $\tau+1$, then executes the rest of the queries in the subsequence.  Policy $\tilde{\vu}$ executes the two queries of $i$ after executing the other queries in the subsequence.  A visual comparison of these policies is provided in Figure \ref{fig: policy_comparison}.  Formally,
\begin{align*}
\hat{\vu}&=\hat{u}(t)=\begin{cases}
u(t), & t \leq \tau \ \mathrm{or} \ t > \tau+\delta \\
u(t-1), & t=\tau+1,\tau+2,\ldots,\tau+\delta, 
\end{cases}\\
\tilde{\vu} &= \tilde{u}(t) = \begin{cases}
u(t), & t < \tau \ \mathrm{or} \ t \geq \tau+\delta \\
u(t+1),& t=\tau,\tau+1,\ldots,\tau+\delta-1. 
\end{cases}
\end{align*}
Let $C$ be the number of non-target queries (i.e., the cost) when using policy $\vu$, $\hat{C}$ be the same quantity when using policy $\hat{\vu}$, and $\tilde{C}$ be the same quantity when using policy $\tilde{\vu}$.
Conditioned on not having found a red marble, it follows that the state trajectories,
\begin{align*}
\hat{x}_{j}(t) &= \begin{cases}
x_{j}(t), & t \leq \tau+1 \ \mathrm{or} \ t > \tau+\delta \\
x_{j}(t-1), & t=\tau+2,\ldots,\tau+\delta , \ j \neq i\\
x_{j}(t-1)+1, & t=\tau+2,\ldots,\tau+\delta , \ j = i\\
\end{cases}\\
\tilde{x}_{j}(t) &= \begin{cases}
x_{j}(t), & t <= \tau \ \mathrm{or} \ t \geq \tau+\delta \\
x_{j}(t+1), & t=\tau+1,\ldots,\tau+\delta-1 , \ j \neq i\\
x_{j}(t+1) - 1, & t=\tau+1,\ldots,\tau+\delta-1 , \ j = i,\\
\end{cases}
\end{align*}
where $\hat{x}_{j}(t)$ and $\tilde{x}_{j}(t)$ are the numbers of times urn $j$ has been queried before stage $t$ under policies $\hat{\vu}$ and $\tilde{\vu}$, respectively.

Using equation \eqref{eq: fullcost}, our assumptions imply that
{\footnotesize
\begin{align*}
& \E[\hat{C}]-\E[C] \geq 0 \\
&  
\sum_{t=0}^{N-1}
\sum_{S \subseteq \mathcal{V}}
	(-1)^{|S|}
	\varphi_{S} 
	\prod_{j \in S} 
	\frac{
		\hat{x}_{j}(t+1)
	}{
		N_{j}
	}
-
\sum_{t=0}^{N-1}
\sum_{S \subseteq \mathcal{V}}
	(-1)^{|S|}
	\varphi_{S} 
	\prod_{j \in S} 
	\frac{
		x_{j}(t+1)
	}{
		N_{j}
	}
	\geq 0
\\
& 
\sum_{S \subseteq \mathcal{V} \setminus \{i\}}
\left(
\sum_{t=\tau+2}^{\tau+\delta}
	(-1)^{|S|}
	\varphi_{S} 
	\prod_{j \in S} 
	\frac{
		\hat{x}_{j}(t)
	}{
		N_{j}
	}
-
\sum_{t=\tau+2}^{\tau+\delta}
	(-1)^{|S|}
	\varphi_{S} 
	\prod_{j \in S} 
	\frac{
		x_{j}(t)
	}{
		N_{j}
	}
\right) \\
& \quad 
+
\sum_{S \subseteq \mathcal{V}: i \in S}
\left(
\sum_{t=\tau+2}^{\tau+\delta}
	(-1)^{|S|}
	\varphi_{S} 
	\prod_{j \in S} 
	\frac{
		\hat{x}_{j}(t)
	}{
		N_{j}
	}
-
\sum_{t=\tau+2}^{\tau+\delta}
	(-1)^{|S|}
	\varphi_{S} 
	\prod_{j \in S} 
	\frac{
		x_{j}(t)
	}{
		N_{j}
	}
\right)
\geq 0
\\
& 
\sum_{S \subseteq \mathcal{V} \setminus \{i\}}
\left(
\sum_{t=\tau+1}^{\tau+\delta-1}
	(-1)^{|S|}
	\varphi_{S} 
	\prod_{j \in S} 
	\frac{
		x_{j}(t)
	}{
		N_{j}
	}
-
\sum_{t=\tau+2}^{\tau+\delta}
	(-1)^{|S|}
	\varphi_{S} 
	\prod_{j \in S} 
	\frac{
		x_{j}(t)
	}{
		N_{j}
	}
\right) \\
& \quad 
+
\sum_{S \subseteq \mathcal{V}: i \in S}
\left(
\sum_{t=\tau+1}^{\tau+\delta-1}
	(-1)^{|S|}
	\varphi_{S} 
	\left(
		\frac{x_{i}(t)+1}{N_{i}}
	\right)
	\prod_{j \in S \setminus \{i\}} 
	\frac{
		x_{j}(t)
	}{
		N_{j}
	}
-
\sum_{t=\tau+2}^{\tau+\delta}
	(-1)^{|S|}
	\varphi_{S} 
	\left(
		\frac{x_{i}(t)}{N_{i}}
	\right)
	\prod_{j \in S \setminus \{i\}} 
	\frac{
		x_{j}(t)
	}{
		N_{j}
	}
\right)
\geq 0
\\
& 
\sum_{S \subseteq \mathcal{V} \setminus \{i\}}
\left(
	(-1)^{|S|}
	\varphi_{S} 
	\prod_{j \in S} 
	\frac{
		x_{j}(\tau+1)
	}{
		N_{j}
	}
-
	(-1)^{|S|}
	\varphi_{S} 
	\prod_{j \in S} 
	\frac{
		x_{j}(\tau+\delta)
	}{
		N_{j}
	}
\right) \\
& \quad 
+
\sum_{S \subseteq \mathcal{V}: i \in S}
    \left(\frac{x_{i}(\tau+\delta)}{N_{i}}\right)
    \left( 
    	(-1)^{|S|}
    	\varphi_{S} 
    	\prod_{j \in S \setminus \{i\}} 
    	\frac{
    		x_{j}(\tau+1)
    	}{
    		N_{j}
    	}
    -
    	(-1)^{|S|}
    	\varphi_{S} 
    	\prod_{j \in S \setminus \{i\}} 
    	\frac{
    		x_{j}(\tau+\delta)
    	}{
    		N_{j}
    	}
    \right) \\
    & \qquad
    +
    	\frac{1}{N_{i}}
    \sum_{S \subseteq \mathcal{V}: i \in S}
    \
    \sum_{t=\tau+1}^{\tau+\delta-1}
    	(-1)^{|S|}
    	\varphi_{S} 
    	\prod_{j \in S \setminus \{i\}} 
    	\frac{
    		x_{j}(t)
    	}{
    		N_{j}
    	}
\geq 0
\\
& 
\sum_{S \subseteq \mathcal{V} \setminus \{i\}}
\left(
	(-1)^{|S|}
	\varphi_{S} 
	\prod_{j \in S} 
	\frac{
		x_{j}(\tau+\delta)
	}{
		N_{j}
	}
-
	(-1)^{|S|}
	\varphi_{S} 
	\prod_{j \in S} 
	\frac{
		x_{j}(\tau+1)
	}{
		N_{j}
	}
\right) \\
& \quad 
+
\sum_{S \subseteq \mathcal{V}: i \in S}
    \left(\frac{x_{i}(\tau+\delta)}{N_{i}}\right)
    \left( 
    	(-1)^{|S|}
    	\varphi_{S} 
    	\prod_{j \in S \setminus \{i\}} 
    	\frac{
    		x_{j}(\tau+\delta)
    	}{
    		N_{j}
    	}
    -
    	(-1)^{|S|}
    	\varphi_{S} 
    	\prod_{j \in S \setminus \{i\}} 
    	\frac{
    		x_{j}(\tau+1)
    	}{
    		N_{j}
    	}
    \right) \\
    & \qquad
    -
    \frac{1}{N_{i}}
    \sum_{S \subseteq \mathcal{V}: i \in S}
    \
    \sum_{t=\tau+2}^{\tau+\delta-1}
    	(-1)^{|S|}
    	\varphi_{S} 
    	\prod_{j \in S \setminus \{i\}} 
    	\frac{
    		x_{j}(t)
    	}{
    		N_{j}
    	}
\leq  
	\frac{1}{N_{i}}
\sum_{S \subseteq \mathcal{V}: i \in S}    
    	(-1)^{|S|}
    	\varphi_{S} 
    	\prod_{j \in S \setminus \{i\}} 
    	\frac{
    		x_{j}(\tau+1)
    	}{
    		N_{j}
    	}.
\end{align*}
}

Optimality of $\vu$ likewise implies
{\footnotesize
\begin{align*}
& \E[\tilde{C}]-\E[C] \geq 0 \\
& 
\sum_{t=0}^{N-1}
\sum_{S \subseteq \mathcal{V}}
	(-1)^{|S|}
	\varphi_{S} 
	\prod_{j \in S} 
	\frac{
		\tilde{x}_{j}(t+1)
	}{
		N_{j}
	}
-
\sum_{t=0}^{N-1}
\sum_{S \subseteq \mathcal{V}}
	(-1)^{|S|}
	\varphi_{S} 
	\prod_{j \in S} 
	\frac{
		x_{j}(t+1)
	}{
		N_{j}
	}
	\geq 0
\\
& 
\sum_{S \subseteq \mathcal{V}}
\left(
\sum_{t=\tau+1}^{\tau+\delta-1}
	(-1)^{|S|}
	\varphi_{S} 
	\prod_{j \in S} 
	\frac{
		\tilde{x}_{j}(t)
	}{
		N_{j}
	}
-
\sum_{t=\tau+1}^{\tau+\delta-1}
	(-1)^{|S|}
	\varphi_{S} 
	\prod_{j \in S} 
	\frac{
		x_{j}(t)
	}{
		N_{j}
	}
\right)
\geq 0
\\
& 
\sum_{S \subseteq \mathcal{V} \setminus \{i\}}
\left(
\sum_{t=\tau+1}^{\tau+\delta-1}
	(-1)^{|S|}
	\varphi_{S} 
	\prod_{j \in S} 
	\frac{
		\tilde{x}_{j}(t)
	}{
		N_{j}
	}
-
\sum_{t=\tau+1}^{\tau+\delta-1}
	(-1)^{|S|}
	\varphi_{S} 
	\prod_{j \in S} 
	\frac{
		x_{j}(t)
	}{
		N_{j}
	}
\right) \\
& \quad 
+
\sum_{S \subseteq \mathcal{V}: i \in S}
\left(
\sum_{t=\tau+1}^{\tau+\delta-1}
	(-1)^{|S|}
	\varphi_{S} 
	\prod_{j \in S} 
	\frac{
		\tilde{x}_{j}(t)
	}{
		N_{j}
	}
-
\sum_{t=\tau+1}^{\tau+\delta-1}
	(-1)^{|S|}
	\varphi_{S} 
	\prod_{j \in S} 
	\frac{
		x_{j}(t)
	}{
		N_{j}
	}
\right)
\geq 0
\\
& 
\sum_{S \subseteq \mathcal{V} \setminus \{i\}}
\left(
\sum_{t=\tau+2}^{\tau+\delta}
	(-1)^{|S|}
	\varphi_{S} 
	\prod_{j \in S} 
	\frac{
		x_{j}(t)
	}{
		N_{j}
	}
-
\sum_{t=\tau+1}^{\tau+\delta-1}
	(-1)^{|S|}
	\varphi_{S} 
	\prod_{j \in S} 
	\frac{
		x_{j}(t)
	}{
		N_{j}
	}
\right) \\
& \quad 
+
\sum_{S \subseteq \mathcal{V}: i \in S}
\left(
\sum_{t=\tau+2}^{\tau+\delta}
	(-1)^{|S|}
	\varphi_{S} 
	\left(
		\frac{x_{i}(t)-1}{N_{i}}
	\right)
	\prod_{j \in S \setminus \{i\}} 
	\frac{
		x_{j}(t)
	}{
		N_{j}
	}
-
\sum_{t=\tau+1}^{\tau+\delta-1}
	(-1)^{|S|}
	\varphi_{S} 
	\left(
		\frac{x_{i}(t)}{N_{i}}
	\right)
	\prod_{j \in S \setminus \{i\}} 
	\frac{
		x_{j}(t)
	}{
		N_{j}
	}
\right)
\geq 0
\\
& 
\sum_{S \subseteq \mathcal{V} \setminus \{i\}}
\left(
	(-1)^{|S|}
	\varphi_{S} 
	\prod_{j \in S} 
	\frac{
		x_{j}(\tau+\delta)
	}{
		N_{j}
	}
-
	(-1)^{|S|}
	\varphi_{S} 
	\prod_{j \in S} 
	\frac{
		x_{j}(\tau+1)
	}{
		N_{j}
	}
\right) \\
& \quad 
+
\sum_{S \subseteq \mathcal{V}: i \in S}
	(-1)^{|S|}
	\varphi_{S} 
	\left(\frac{x_{i}(\tau+\delta)-1}{N_{i}}\right)
	\prod_{j \in S \setminus \{i\}} 
	\frac{
		x_{j}(\tau+\delta)
	}{
		N_{j}
	}
-
	(-1)^{|S|}
	\varphi_{S} 
	\left(\frac{x_{i}(\tau+1)}{N_{i}}\right)
	\prod_{j \in S \setminus \{i\}} 
	\frac{
		x_{j}(\tau+1)
	}{
		N_{j}
	}
\\ & \qquad
-
	\frac{1}{N_{i}}
\sum_{S \subseteq \mathcal{V}: i \in S} 
\
\sum_{t=\tau+2}^{\tau+\delta-1}
	(-1)^{|S|}
	\varphi_{S} 
	\prod_{j \in S \setminus \{i\}} 
	\frac{
		x_{j}(t)
	}{
		N_{j}
	}
\geq 0
\\
& 
\sum_{S \subseteq \mathcal{V} \setminus \{i\}}
\left(
	(-1)^{|S|}
	\varphi_{S} 
	\prod_{j \in S} 
	\frac{
		x_{j}(\tau+\delta)
	}{
		N_{j}
	}
-
	(-1)^{|S|}
	\varphi_{S} 
	\prod_{j \in S} 
	\frac{
		x_{j}(\tau+1)
	}{
		N_{j}
	}
\right) \\
& \quad 
+
\sum_{S \subseteq \mathcal{V}: i \in S}
    	\frac{x_{i}(\tau+\delta)}{N_{i}}
    \left( 
    	(-1)^{|S|}
    	\varphi_{S} 
    	\prod_{j \in S \setminus \{i\}} 
    	\frac{
    		x_{j}(\tau+\delta)
    	}{
    		N_{j}
    	}
    -
    	(-1)^{|S|}
    	\varphi_{S} 
    	\prod_{j \in S \setminus \{i\}} 
    	\frac{
    		x_{j}(\tau+1)
    	}{
    		N_{j}
    	}
    \right) \\
    & \qquad
    -
    	\frac{1}{N_{i}}
    \sum_{S \subseteq \mathcal{V}: i \in S} 
    \
    \sum_{t=\tau+2}^{\tau+\delta-1}
    	(-1)^{|S|}
    	\varphi_{S} 
    	\prod_{j \in S \setminus \{i\}} 
    	\frac{
    		x_{j}(t)
    	}{
    		N_{j}
    	}
\geq 
\left(\frac{1}{N_{i}}\right)
\sum_{S \subseteq \mathcal{V}: i \in S}
    	(-1)^{|S|}
    	\varphi_{S} 
    	\prod_{j \in S \setminus \{i\}} 
    	\frac{
    		x_{j}(\tau+\delta)
    	}{
    		N_{j}
    	}.
\end{align*}
}

Now let 
\begin{align*}
\alpha &  =
\sum_{S \subseteq \mathcal{V} \setminus \{i\}}
\left(
	(-1)^{|S|}
	\varphi_{S} 
	\prod_{j \in S} 
	\frac{
		x_{j}(\tau+\delta)
	}{
		N_{j}
	}
-
	(-1)^{|S|}
	\varphi_{S} 
	\prod_{j \in S} 
	\frac{
		x_{j}(\tau+1)
	}{
		N_{j}
	}
\right) \\
& \quad 
+
\sum_{S \subseteq \mathcal{V}: i \in S}
    \left(\frac{x_{i}(\tau+\delta)}{N_{i}}\right)
    \left( 
    	(-1)^{|S|}
    	\varphi_{S} 
    	\prod_{j \in S \setminus \{i\}} 
    	\frac{
    		x_{j}(\tau+\delta)
    	}{
    		N_{j}
    	}
    -
    	(-1)^{|S|}
    	\varphi_{S} 
    	\prod_{j \in S \setminus \{i\}} 
    	\frac{
    		x_{j}(\tau+1)
    	}{
    		N_{j}
    	}
    \right) \\
    & \qquad
    -
        \frac{1}{N_{i}}
    \sum_{S \subseteq \mathcal{V}: i \in S} 
    \
    \sum_{t=\tau+2}^{\tau+\delta-1}
    	(-1)^{|S|}
    	\varphi_{S} 
    	\prod_{j \in S \setminus \{i\}} 
    	\frac{
    		x_{j}(t)
    	}{
    		N_{j}
    	}
,
\end{align*}
which appears in both of the expected cost inequalities we have derived.
We have established that 
\begin{align*}
\alpha &\geq \left(\frac{1}{N_{i}}\right)
\sum_{S \subseteq \mathcal{V}: i \in S}
    	(-1)^{|S|}
    	\varphi_{S} 
    	\prod_{j \in S \setminus \{i\}} 
    	\frac{
    		x_{j}(\tau+\delta)
    	}{
    		N_{j}
    	}
\\
\alpha &\leq \left(\frac{1}{N_{i}}\right)
\sum_{S \subseteq \mathcal{V}: i \in S}
    	(-1)^{|S|}
    	\varphi_{S} 
    	\prod_{j \in S \setminus \{i\}} 
    	\frac{
    		x_{j}(\tau+1)
    	}{
    		N_{j}
    	}
\end{align*}

We now show that for any $i \in \mathcal{V}$, for any $u(t) \in \mathcal{V} \setminus \{i\}$,
\begin{equation}
h(t) = \sum_{S\subseteq \mathcal{V}: \; i \in S}
    	(-1)^{|S|}
    	\varphi_{S} 
    	\prod_{j \in S \setminus \{i\}} 
    	\frac{
    		x_{j}(t)
    	}{
    		N_{j}
    	} \label{eq: ht}
\end{equation}
is a nondecreasing function that is strictly increasing when $\varphi_{\{i,u(t)\}}>0$.

Observe that 
\begin{align*}
h(t+1) &= \sum_{S\subseteq \mathcal{V}: \; i \in S}
    	(-1)^{|S|}
    	\varphi_{S} 
    	\prod_{j \in S \setminus \{i\}} 
    	\frac{
    		x_{j}(t+1)
    	}{
    		N_{j}
    	}
\\
&= \sum_{S\subseteq \mathcal{V} \setminus \{u(t)\}: \; i \in S}
    	(-1)^{|S|}
    	\varphi_{S} 
    	\prod_{j \in S \setminus \{i\}} 
    	\frac{
    		x_{j}(t)
    	}{
    		N_{j}
    	}
\\
& \qquad +
\left(\frac{x_{u(t)}(t)+1}{N_{u(t)}}\right)
\sum_{S\subseteq \mathcal{V}: \; \{i,u(t)\} \subseteq S}
    	(-1)^{|S|}
    	\varphi_{S} 
    	\prod_{j \in S \setminus \{i,u(t)\}} 
    	\frac{
    		x_{j}(t)
    	}{
    		N_{j}
    	}	
\\
& = h(t) +\left( \frac{1}{N_{u(t)}} \right)
\sum_{S\subseteq \mathcal{V}: \; \{i,u(t)\} \subseteq S}
    	(-1)^{|S|}
    	\varphi_{S} 
    	\prod_{j \in S \setminus \{i,u(t)\}} 
    	\frac{
    		x_{j}(t)
    	}{
    		N_{j}
    	}.	
\end{align*}

From Lemma \ref{lem: problemma}, 
\begin{align*}
&
\left(
	\frac{x_{i}(t)x_{u(t)}(t)}{N_{i}N_{u(t)}}
\right)
\sum_{S\subseteq \mathcal{V}: \; \{i,u(t)\} \subseteq S}
    	(-1)^{|S|}
    	\varphi_{S} 
    	\prod_{j \in S \setminus \{i,u(t)\}} 
    	\frac{
    		x_{j}(t)
    	}{
    		N_{j}
    	}
	\geq 0
	\\
& \quad\Rightarrow
\left(
	\frac{1}{N_{u(t)}}
\right)
\sum_{S\subseteq \mathcal{V}: \; \{i,u(t)\} \subseteq S}
    	(-1)^{|S|}
    	\varphi_{S} 
    	\prod_{j \in S \setminus \{i,u(t)\}} 
    	\frac{
    		x_{j}(t)
    	}{
    		N_{j}
    	}
	\geq 0.
\end{align*}
This final inequality implies
$h(t+1) \geq h(t)$.  Because urn $i$ is not queried in stages $t=\tau+1,\ldots,\tau+\delta$, $h(t)$ is nondecreasing over these stages.  Therefore we can write
\begin{equation}
\alpha \geq h(\tau+\delta) \geq h(\tau+1) \geq \alpha.
\label{eq: hineq}
\end{equation}

This expression can only be satisfied by equality, which means that for any optimal policy that is not a block policy, we can maintain optimality while successively permuting the policy so that the urn queries are arranged into blocks.  
\halmos
\end{proof}

\end{proof}

\subsection{Theorem \ref{thm: independence}}
\label{proof: independence}
\begin{proof}{Proof of Theorem \ref{thm: independence}.}
\begin{proof}{}
This result comes from substituting the appropriate products into the result from Theorem \ref{thm: genprob}.  For any urn $i \in \mathcal{V}$,
\begin{align*}
\P(A_{i}|\vx(t)) &= 
\frac{
		\left(
			1-\frac{
					x_{i}(t)
				}{
					N_{i}
				}
		\right)
	\sum_{
		\{S \subseteq \mathcal{V}: \; i \in S\}
	}
	(-1)^{|S|-1}
	\varphi_{S}
	\prod_{
		j \in S \setminus \{i\}
	}
	\frac{
		x_{j}(t)
	}{
		N_{j}
	}
}{
	\sum_{S \subseteq \mathcal{V}}
	(-1)^{|S|}
	\varphi_{S} 
	\prod_{j \in S} 
	\frac{
		x_{j}(t)
	}{
		N_{j}
	}
}
\\
& = 
\frac{
		\varphi_{i}
		\left(
			1-\frac{
					x_{i}(t)
				}{
					N_{i}
				}
		\right)
	\sum_{
		\{S \subseteq \mathcal{V} \setminus \{i\}\}
	}
	(-1)^{|S|}
	\prod_{
		j \in S
 	}
	\varphi_{j}
	\frac{
		x_{j}(t)
	}{
		N_{j}
	}
}{
	\sum_{S \subseteq \mathcal{V}}
	(-1)^{|S|}
	\prod_{j \in S} 
	\varphi_{j}
	\frac{
		x_{j}(t)
	}{
		N_{j}
	}
}
\\
& = 
\frac{
		\varphi_{i}
		\left(
			1-\frac{
					x_{i}(t)
				}{
					N_{i}
				}
		\right)
	\sum_{
		\{S \subseteq \mathcal{V} \setminus \{i\}\}
	}
	(-1)^{|S|}
	\prod_{
		j \in S
 	}
	\varphi_{j}
	\frac{
		x_{j}(t)
	}{
		N_{j}
	}
}{
	\sum_{S \subseteq \mathcal{V} \setminus\{i\}}
	(-1)^{|S|}
	\prod_{j \in S} 
	\varphi_{j}
	\frac{
		x_{j}(t)
	}{
		N_{j}
	}
+
	\sum_{S \subseteq \mathcal{V}: i \in S}
	(-1)^{|S|}
	\prod_{j \in S} 
	\varphi_{j}
	\frac{
		x_{j}(t)
	}{
		N_{j}
	}
}
\\
& = 
\frac{
		\varphi_{i}
		\left(
			1-\frac{
					x_{i}(t)
				}{
					N_{i}
				}
		\right)
	\sum_{
		\{S \subseteq \mathcal{V} \setminus \{i\}\}
	}
	(-1)^{|S|}
	\prod_{
		j \in S
 	}
	\varphi_{j}
	\frac{
		x_{j}(t)
	}{
		N_{j}
	}
}{
	\sum_{S \subseteq \mathcal{V} \setminus\{i\}}
	(-1)^{|S|}
	\prod_{j \in S} 
	\varphi_{j}
	\frac{
		x_{j}(t)
	}{
		N_{j}
	}
-
\varphi_{i}\frac{x_{i}(t)}{N_{i}}
	\sum_{S \subseteq \mathcal{V} \setminus\{i\}}
	(-1)^{|S|}
	\prod_{j \in S} 
	\varphi_{j}
	\frac{
		x_{j}(t)
	}{
		N_{j}
	}
}
\\
& = 
\frac{
	\varphi_{i}
	\left(
            	1-\frac{
            		x_{i}(t)
            	}
            	{
            		N_{i}
            	}
	\right)
}{
	1-\varphi_{i}
	\left(
		\frac{
			x_{i}(t)
		}
		{
			N_{i}
		}
	\right)
}.
\end{align*}

In a similar manner we use Theorem \ref{thm: genprob} to find the urn probability for  subset $U \subseteq \mathcal{V}$ at stage $t$,
\begin{align*}
\P\left(\bigcap_{i \in U}A_{i}\right) &= 
\frac{
	\left[
		\prod_{i \in U}
		\left(
			1-\frac{
					x_{i}(t)
				}{
					N_{i}
				}
		\right)
	\right]
	\sum_{
		\{S \subseteq \mathcal{V}: \; S \supseteq U\}
	}
	(-1)^{|S|-|U|}\varphi_{S}
	\prod_{
		j \in S \setminus U
	}
	\frac{
		x_{j}(t)
	}{
		N_{j}
	}
}{
	\sum_{S \subseteq \mathcal{V}}
	(-1)^{|S|}
	\varphi_{S} 
	\prod_{j \in S} 
	\frac{
		x_{j}(t)
	}{
		N_{j}
	}
}
\\
&=\frac{
	\left[
		\prod_{i \in U}
		\varphi_{i}
		\left(
			1-\frac{
					x_{i}(t)
				}{
					N_{i}
				}
		\right)
	\right]
	\sum_{
		S \subseteq \mathcal{V} \setminus U
	}
	(-1)^{|S|}
	\prod_{
		j \in S
	}
	\varphi_{j}
	\frac{
		x_{j}(t)
	}{
		N_{j}
	}
}{
	\sum_{T \subseteq U}
	\sum_{S \subseteq \mathcal{V} \setminus U}
	(-1)^{|S|+|T|}
	\prod_{j \in S \cup T} 
	\varphi_{j} 
	\frac{
		x_{j}(t)
	}{
		N_{j}
	}
}
\\
&=\frac{
	\left[
		\prod_{i \in U}
		\varphi_{i}
		\left(
			1-\frac{
					x_{i}(t)
				}{
					N_{i}
				}
		\right)
	\right]
	\sum_{
		S \subseteq \mathcal{V} \setminus U
	}
	(-1)^{|S|}
	\prod_{
		j \in S
	}
	\varphi_{j}
	\frac{
		x_{j}(t)
	}{
		N_{j}
	}
}{
	\sum_{T \subseteq U}
	(-1)^{|T|}
	\prod_{k \in T}
	\varphi_{k} 
	\frac{
		x_{k}(t)
	}{
		N_{k}
	}
	\sum_{S \subseteq \mathcal{V} \setminus U}
	(-1)^{|S|}
	\prod_{j \in S} 
	\varphi_{j} 
	\frac{
		x_{j}(t)
	}{
		N_{j}
	}
}
\\
&=\frac{
	\left[
		\prod_{i \in U}
		\varphi_{i}
		\left(
			1-\frac{
					x_{i}(t)
				}{
					N_{i}
				}
		\right)
	\right]
}{
	\sum_{S \subseteq U}
	(-1)^{|S|}
	\prod_{j \in S}
	\varphi_{j} 
	\frac{
		x_{j}(t)
	}{
		N_{j}
	}
} = \prod_{i \in U} \P(A_{i}).
\end{align*}
In the final equality, we have used the property that for any $\beta_{1},\ldots,\beta_{M}$,
\[
\prod_{i=1}^{M} (1-\beta_{i}) = \sum_{S \subseteq [M]} \prod_{j \in S} (-\beta_{j}).
\]
We can verify this property by induction.  Define $\prod_{j \in \emptyset} (-\beta_{j}) = 1$.  Now observe
\begin{align*}
\prod_{i=1}^{M+1}(1-\beta_{i}) & = (1-\beta_{M+1})\prod_{i=1}^{M}(1-\beta_{i}) \\
& = (1-\beta_{M+1}) \sum_{S \subseteq [M]} \prod_{j \in S} (-\beta_{j}) \\
& = \sum_{S \subseteq [M]} \prod_{j \in S} (-\beta_{j}) -\beta_{M+1}\sum_{S \subseteq [M]} \prod_{j \in S} (-\beta_{j}) \\
& = \sum_{S \subseteq [M+1]} \prod_{j \in S} (-\beta_{j}). \\
\end{align*}
Setting $\beta_{i}=\varphi_{i}\left(\frac{x_{i}(t)}{N_{i}}\right)$ achieves the desired result. \halmos
\end{proof}

\end{proof}

\subsection{Theorem \ref{thm: independence_optimality}}
\label{proof: independence_optimality}
\begin{proof}{Proof of Theorem \ref{thm: independence_optimality}.}
\begin{proof}{}
First we prove that the condition in 
equation \ref{eq: independence_optimality} implies optimality by contrapositive.  Let $\mathbf{u}_{B}=(v^{1},v^{2},\ldots,v^{|\mathcal{V}|})$ be a block policy that does not satisfy this condition, and let $i$ be an index for which
\[
N_{v^{i}}\left(\frac{2-\varphi_{v^{i}}}{\varphi_{v^{i}}}\right) >
N_{v^{i+1}}\left(\frac{2-\varphi_{v^{i+1}}}{\varphi_{v^{i+1}}}\right).
\]
Also, we define the first stage in which urn $v^{i}$ is queried in this policy as $\tau=\sum_{j=0}^{i-1} N_{j}$. 

We now construct an alternative block policy $\tilde{\mathbf{u}}_{B}$, so that
\[
\tilde{v}^{j}=\begin{cases}
v^{j} & j \notin \{i,i+1\} \\
v^{i+1} & j = i \\
v^{i} & j = i+1.
\end{cases}
\]
Let $\E[C]=\sum_{k=0}^{N-1}\prod_{t=0}^{k}\P(w(\vx(t),u(t))=0)$ be the expected cost of policy $\mathbf{u}_{B}$ and 
$\E[\tilde{C}]=\sum_{k=0}^{N-1}\prod_{t=0}^{k}\P(w(\tilde{\vx}(t),\tilde{u}(t))=0)$ be the expected cost of policy $\tilde{\mathbf{u}}_{B}$.  For brevity, we also define
$
\gamma = \prod_{t=0}^{\tau-1}P(w(\vx(t),u(t))=0) = \prod_{j=1}^{i-1}(1-\varphi_{v^{j}}) > 0,
$
according to Corollary \ref{cor: independence_block}.
Now consider the difference in expected cost,
\begin{align}
\E[C]-\E[\tilde{C}] &= \sum_{k=0}^{N-1}\prod_{t=0}^{k}\P(w(\vx(t),u(t))=0) - \sum_{k=0}^{N-1}\prod_{t=0}^{k}\P(w(\tilde{\vx}(t),\tilde{u}(t))=0) \nonumber 
\\
& = \sum_{k=\tau}^{\tau+N_{v^{i}}+N_{v^{i+1}}-1}\prod_{t=0}^{k}\P(w(\vx(t),u(t))=0) \nonumber
\\
& \quad- \sum_{k=\tau}^{\tau+N_{v^{i}}+N_{v^{i+1}}-1}\prod_{t=0}^{k}\P(w(\tilde{\vx}(t),\tilde{u}(t))=0) \nonumber
\\
& 
\stackeq{a}
 \gamma \left(\sum_{k=\tau}^{\tau+N_{v^{i}}+N_{v^{i+1}}-1}\prod_{t=\tau}^{k}\P(w(\vx(t),u(t))=0) \right) \nonumber
\\
& \qquad - 
\gamma \left(\sum_{k=\tau}^{\tau+N_{v^{i}}+N_{v^{i+1}}-1}\prod_{t=\tau}^{k}\P(w(\tilde{\vx}(t),\tilde{u}(t))=0)
\right) \nonumber
\\
& 
\stackeq{b} 
\gamma
\left(
	\left(
		N_{v^{i}}-\frac{(N_{v^{i}}+1)\varphi_{v^{i}}}{2}
	\right) 
	+ (1-\varphi_{v^{i}})
	\left(
		N_{v^{i+1}}-\frac{(N_{v^{i+1}}+1)\varphi_{v^{i+1}}}{2}
	\right)
\right) \nonumber
\\
& \quad -
\gamma
\left(
	\left(
		N_{v^{i+1}}-\frac{(N_{v^{i+1}}+1)\varphi_{v^{i+1}}}{2}
	\right) 
	+ (1-\varphi_{v^{i+1}})
	\left(
		N_{v^{i}}-\frac{(N_{v^{i}}+1)\varphi_{v^{i}}}{2}
	\right)
\right) \nonumber
\\
& =
\gamma
\left(
	\varphi_{v^{i+1}}
	\left(
		N_{v^{i}}-\frac{(N_{v^{i}}+1)\varphi_{v^{i}}}{2}
	\right) 
	-
	\varphi_{v^{i}}
	\left(
		N_{v^{i+1}}-\frac{(N_{v^{i+1}}+1)\varphi_{v^{i+1}}}{2}
	\right) 	
\right)  \nonumber
\\
& = 
\gamma\varphi_{v^{i}}\varphi_{v^{i+1}}
\left(
	\left(
		\frac{2N_{v^{i}} - N_{v^{i}} \varphi_{v^{i}}-\varphi_{v^{i}}}{2\varphi_{v^{i}}}
	\right) 
	-
	\left(
		\frac{2N_{v^{i+1}} - N_{v^{i+1}} \varphi_{v^{i+1}}-\varphi_{v^{i}}}{2\varphi_{v^{i+1}}}
	\right) 	
\right)  \nonumber
\\
& 
= 
\frac{\gamma\varphi_{v^{i}}\varphi_{v^{i+1}}}{2}
\left(
	N_{v^{i}}
	\left(
		\frac{2- \varphi_{v^{i}}}{\varphi_{v^{i}}}
	\right) 
	-
	N_{v^{i+1}}
	\left(
		\frac{2 - \varphi_{v^{i+1}}}{\varphi_{v^{i+1}}}
	\right) 	
\right) > 0. \label{eq: independence_contrapositive}
\end{align}
Steps (a) and (b) follow immediately from Corollary \ref{cor: independence_block}.  The difference in expected cost, $\E[C]-\E[\tilde{C}]$, is strictly positive so that block policy $\mathbf{u}_{B}$ cannot be optimal.  Therefore, the condition given in Theorem \ref{thm: independence_optimality} is necessary for optimality.

Now we show that the same condition is sufficient for optimality, i.e., if a block policy satisfies equation \eqref{eq: independence_optimality} in Theorem \ref{thm: independence_optimality}, then it must be an optimal policy.  In order to form a contradiction, suppose now that $\mathbf{u}_{B}$ is a block policy that satisfies the condition but is not optimal.  Let $\mathbf{u}^{\star}_{B}$ be the optimal block policy, which Theorem \ref{thm: blockpolicy} guarantees to exist.  

We know from the above argument that $\mathbf{u}^{\star}_{B}$ also must satisfy the condition, which implies that policies $\mathbf{u}_{B}$ and $\mathbf{u}^{\star}_{B}$ can only differ by permuting subsequences $v^{i},v^{i+1},\ldots v^{i+\delta}$ for which
\[
N_{v^{i}}
\left(
	\frac{2-\varphi_{v^{i}}}{\varphi_{v^{i}}}
\right) 
=
N_{v^{i+1}}
\left(
	\frac{2-\varphi_{v^{i+1}}}{\varphi_{v^{i+1}}}
\right) 
=\cdots = N_{v^{i+\delta}}
\left(
	\frac{2-\varphi_{v^{i+\delta}}}{\varphi_{v^{i+\delta}}}
\right)
.
\]
The optimal block policy $\mathbf{u}^{\star}_{B}$ therefore can be constructed by executing a finite number of sequential pairwise exchanges in the urn ordering in block policy $\mathbf{u}_{B}$, each satisfying the condition with equality.  However, it follows from the inequality in equation \eqref{eq: independence_contrapositive} that any such permutation results in the same expected policy cost.  We can conclude that the expected costs of the two policies are equal, establishing the contradiction and showing $\mathbf{u}_{B}$ to be an optimal policy. \halmos
\end{proof}

\end{proof}

\subsection{Theorem \ref{thm: single_marble_optimality}}
\label{proof: single_marble_optimality}
\begin{proof}{Proof of Theorem \ref{thm: single_marble_optimality}.}
\begin{proof}{}
The proof is similar to that of Theorem \ref{thm: independence_optimality}.  First we prove that the condition in equation \eqref{eq: single_marble_optimality} implies optimality by contrapositive.  Let $\mathbf{u}_{B}=(v^{1},v^{2},\ldots,v^{|\mathcal{V}|})$ be a block policy that does not satisfy equation \eqref{eq: single_marble_optimality}, and let $i$ be an index for which
\[
\frac{\varphi_{v^{i}}}
{N_{v^{i}}}
< 
\frac{\varphi_{v^{i+1}}}
{N_{v^{i+1}}}.
\]
Also, we define the first stage in which urn $v^{i}$ is queried in this policy as $\tau=\sum_{j=1}^{i-1} N_{j}$. 

We now construct an alternative block policy $\tilde{\mathbf{u}}_{B}$, so that
\[
\tilde{v}^{j}=\begin{cases}
v^{j} & j \notin \{i,i+1\} \\
v^{i+1} & j = i \\
v^{i} & j = i+1.
\end{cases}
\]
Let $\E[C]=\sum_{k=0}^{N-1}\prod_{t=0}^{k}\P(w(\vx(t),u(t))=0)$ be the expected cost of policy $\mathbf{u}_{B}$ and  $\E[\tilde{C}]=\sum_{k=0}^{N-1}\prod_{t=0}^{k}\P(w(\tilde{\vx}(t),\tilde{u}(t))=0)$ be the expected cost of policy $\tilde{\mathbf{u}}_{B}$.  
Now consider the difference in expected cost,
\begin{align*}
\E[C]-\E[\tilde{C}] &= \sum_{k=0}^{N-1}\prod_{t=0}^{k}\P(w(\vx(t),u(t))=0) - \sum_{k=0}^{N-1}\prod_{t=0}^{k}\P(w(\tilde{\vx}(t),\tilde{u}(t))=0) \\
& = \sum_{k=\tau}^{\tau+N_{v^{i}}+N_{v^{i+1}}-1}\prod_{t=0}^{k}\P(w(\vx(t),u(t))=0) 
\\
& \quad- \sum_{k=\tau}^{\tau+N_{v^{i}}+N_{v^{i+1}}-1}\prod_{t=0}^{k}\P(w(\tilde{\vx}(t),\tilde{u}(t))=0) 
\\
& 
\stackeq{c}
 \left(
	N_{v^{i}}-\frac{(N_{v^{i}}+1)\varphi_{v^{i}}}{2}
	-N_{v^{i}}\sum_{j=1}^{i-1}\varphi_{v^{j}} 
	+	N_{v^{i+1}}-\frac{(N_{v^{i+1}}+1)\varphi_{v^{i+1}}}{2}
	-N_{v^{i+1}}\sum_{j=1}^{i}\varphi_{v^{j}}  
\right)
\\
& \quad 
-\Bigg(
	N_{v^{i+1}}-\frac{(N_{v^{i+1}}+1)\varphi_{v^{i+1}}}{2}
	-N_{v^{i+1}}\sum_{j=1}^{i-1}\varphi_{v^{j}} 
	\\
	&\qquad
	+	N_{v^{i}}-\frac{(N_{v^{i}}+1)\varphi_{v^{i}}}{2}
	-N_{v^{i}}\sum_{j=1}^{i-1}\varphi_{v^{j}} - N_{v^{i}}\varphi_{v^{i+1}}  
\Bigg)
\\
& = N_{v^{i}}
\varphi_{v^{i+1}}
-N_{v^{i+1}}
\varphi_{v^{i}} \\
& = N_{v^{i}}
N_{v^{i+1}}
\left(
	\frac{
		\varphi_{
			v^{i+1}
		}
	}
	{
		N_{
			v^{i+1}
		}
	}
	-\frac{
		\varphi_{
			v^{i}
		}
	}
	{
		N_{v^{i}}
	}
\right)  > 0.
\end{align*}
Step (c) follows from substituting the results in Corollary \ref{cor: single_marble_block}.
The difference in expected cost, $\E[C]-\E[\tilde{C}]$, is strictly positive so that block policy $\mathbf{u}_{B}$ cannot be optimal.  Therefore, the condition given in equation \eqref{eq: independence_optimality} is necessary for optimality.

Now we show that the same condition is sufficient for optimality, i.e., if a block policy satisfies equation \eqref{eq: single_marble_optimality}, then it must be an optimal policy.  In order to form a contradiction, suppose now that $\mathbf{u}_{B}$ is a block policy that satisfies the condition but is not optimal.  Let $\mathbf{u}^{\star}_{B}$ be the optimal block policy, which Theorem \ref{thm: blockpolicy} guarantees to exist.  

We know from the above argument that $\mathbf{u}^{\star}_{B}$ also must satisfy the condition, which implies that policies $\mathbf{u}_{B}$ and $\mathbf{u}^{\star}_{B}$ can only differ by permuting subsequences $v^{i},v^{i+1},\ldots v^{i+\delta}$ for which
\[
\frac{\varphi_{v^{i}}}
{N_{v^{i}}}
=
\frac{\varphi_{v^{i+1}}}
{N_{v^{i+1}}}
=\cdots = 
\frac{\varphi_{v^{i+\delta}}}
{N_{v^{i+\delta}}}.
\]
The optimal block policy $\mathbf{u}^{\star}_{B}$ therefore can be constructed by executing a finite number of sequential pairwise exchanges in the urn ordering in block policy $\mathbf{u}_{B}$, in which the two urns in each exchange satisfy the equation \eqref{eq: single_marble_optimality} with equality.  However, it follows from our previous argument that any such permutation does not affect expected policy cost.  We can conclude that the expected costs of the two policies are equal, establishing the contradiction and showing $\mathbf{u}_{B}$ to be an optimal policy.
\end{proof}

\end{proof}

\subsection{Theorem \ref{thm: monotonicity}}
\label{proof: monotonicity}
\begin{proof}{Proof of Theorem \ref{thm: monotonicity}.}
\begin{proof}{}
The first inequality in Theorem \ref{thm: monotonicity} follows  from Equation \ref{eq: probrecursion} in Appendix \ref{proof: genprob}.  Given $u(t) \in U$,
\begin{align*}
\P\left(\bigcap_{i\in U} A_{i}|\vx(t+1)\right) 
&=\P\left(\bigcap_{i\in U} A_{i}|\vx(t)\right)
\left(
\frac{
1-\left(\frac{1}{N_{u(t)}-x_{u(t)}(t)}\right)}
{1-\left(\frac{1}{N_{u(t)}-x_{u(t)}(t)}\right)\P(A_{u(t)}|\vx(t))}
\right) \\
&\leq \P\left(\bigcap_{i\in U} A_{i}|\vx(t)\right).
\end{align*}
Note that if $\P\left(\bigcap_{i\in U} A_{i}|\vx(t)\right)= 0$, then $\P\left(\bigcap_{i\in U} A_{i}|\vx(t+1)\right)= 0$.  Equality is likewise preserved when $\P\left( A_{u(t)}|\vx(t)\right)= 1$. (We intentionally omit the case when no red marble is found after drawing all of the marbles in an urn $i$ for which $\varphi_{i}=1$.)  Assuming $0 < \P\left(\bigcap_{i\in U} A_{i}|\vx(t)\right)$ \emph{and} $ \P\left(A_{u(t)}|\vx(t)\right) < 1$, then the inequality becomes strict:
\begin{align*}
0 <
\left(\frac{1}{N_{u(t)}-x_{u(t)}(t)}\right)\P(A_{u(t)}|\vx(t)) 
&<
\left(\frac{1}{N_{u(t)}-x_{u(t)}(t)}\right)
\leq
1 \\
1-\left(\frac{1}{N_{u(t)}-x_{u(t)}(t)}\right)\P(A_{u(t)}|\vx(t)) 
& > 
1-\left(\frac{1}{N_{u(t)}-x_{u(t)}(t)}\right) 
\\
1 & > 
\frac{1-\left(\frac{1}{N_{u(t)}-x_{u(t)}(t)}\right) }
{1-\left(\frac{1}{N_{u(t)}-x_{u(t)}(t)}\right)\P(A_{u(t)}|\vx(t)) }
\\
\P\left(\bigcap_{i\in U} A_{i}|\vx(t)\right)
& >
\P\left(\bigcap_{i\in U} A_{i}|\vx(t)\right)
\left(
\frac{1-\left(\frac{1}{N_{u(t)}-x_{u(t)}(t)}\right) }
{1-\left(\frac{1}{N_{u(t)}-x_{u(t)}(t)}\right)\P(A_{u(t)}|\vx(t)) }
\right)\\
&=\P\left(\bigcap_{i\in U} A_{i}|\vx(t+1)\right)
\end{align*}

To prove the second inequality in Theorem \ref{thm: monotonicity}, first note that for any real numbers $\alpha,\beta$, such that $
\alpha > 0$, $\alpha+\beta > 0$, and $|\beta| > 0$,
\begin{align*}
\left(\alpha+\beta\right)^{2} = \alpha^{2}+2\alpha\beta+\beta^{2} 
& > \alpha^{2}+2\alpha\beta = \alpha\left(\alpha+2\beta\right)\\
\left(\frac{\alpha + \beta}{\alpha}\right) &> \left(\frac{\alpha+2\beta}{\alpha+\beta}\right)
\end{align*}
Now let
\begin{align*}
\alpha&=\sum_{S \subseteq \mathcal{V}} (-1)^{|S|}\varphi_{S}\prod_{i \in S}\frac{x_{i}(t)}{N_{i}}
\\
\beta & = 
\left(
	\frac{1}{
		N_{u(t)}
	}
\right) 
\sum_{S \subseteq \mathcal{V}: u(t) \in S} 
(-1)^{|S|}\varphi_{S}\prod_{i \in S \setminus \{u(t)\}}\frac{x_{i}(t)}{N_{i}}.
\end{align*}
Observe that 
\begin{align*}
\alpha+\beta &= \sum_{S \subseteq \mathcal{V}} (-1)^{|S|}\varphi_{S}\prod_{i \in S}\frac{x_{i}(t+1)}{N_{i}} \\
\alpha+2\beta &= \sum_{S \subseteq \mathcal{V}} (-1)^{|S|}\varphi_{S}\prod_{i \in S}\frac{x_{i}(t+2)}{N_{i}}, 
\end{align*}
assuming urn $u(t)$ is queried again in stage $t+1$.  From Lemma \ref{lem: problemma} and Corollary \ref{cor: cost}, we can conclude that $\alpha > 0$ and $\alpha+\beta>0$ as long as there is a positive probability of reaching stage $t$ without finding a red marble.  Therefore,
\begin{align*}
\left(\frac{\alpha + \beta}{\alpha}\right) 
&\geq 
\left(\frac{\alpha+2\beta}{\alpha+\beta}\right)
\\
\frac{\sum_{S \subseteq \mathcal{V}} (-1)^{|S|}\varphi_{S}\prod_{i \in S}\frac{x_{i}(t+1)}{N_{i}}
}
{
\sum_{S \subseteq \mathcal{V}} (-1)^{|S|}\varphi_{S}\prod_{i \in S}\frac{x_{i}(t)}{N_{i}}
} 
&\geq
\frac{\sum_{S \subseteq \mathcal{V}} (-1)^{|S|}\varphi_{S}\prod_{i \in S}\frac{x_{i}(t+2)}{N_{i}}
}
{
\sum_{S \subseteq \mathcal{V}} (-1)^{|S|}\varphi_{S}\prod_{i \in S}\frac{x_{i}(t+1)}{N_{i}}
} 
\\
\P\left(w(\vx(t),u(t))=0\right) 
&\geq 
\P\left(w(\vx(t+1),u(t))=0\right)
\\
\P\left(w(\vx(t),u(t))=1\right) 
&\leq 
\P\left(w(\vx(t+1),u(t))=1\right).
\end{align*}

It follows from  Lemma \ref{lem: problemma} that  $\beta=0$ only when there is no probability of drawing a red marble from urn $u(t)$ in stage $t$.  Therefore, the inequality is strict whenever $\P(w(\vx(t),u(t))=1)>0$.
\end{proof}

\end{proof}

\subsection{Theorem \ref{thm: probrate}}
\label{proof: probrate}
\begin{proof}{Proof of Theorem \ref{thm: probrate}.}
\begin{proof}{}
We assume that all probabilities are positive.  Note that the result in Theorem \ref{thm: probrate} can be restated
\begin{equation}
\frac{\P\left(w(\vx(t),i)=1\right)}{\P\left(w(\vx(t),u(t))=1\right)}
\geq
\frac{\P\left(w(\vx(t+1),i)=1\right)}{\P\left(w(\vx(t+1),u(t))=1\right)}. \label{eq: probrate_restated}
\end{equation}

Now let
\begin{align*}
&\gamma_{i}=\sum_{S \subseteq \mathcal{V} \setminus \{u(t)\}: \; i \in S} (-1)^{|S|}\varphi_{S}\prod_{i \in S \setminus \{i\}}\frac{x_{i}(t)}{N_{i}}\\
&\gamma_{u(t)}=\sum_{S \subseteq \mathcal{V} \setminus \{i\}: \; u(t) \in S} (-1)^{|S|}\varphi_{S}\prod_{i \in S \setminus \{u(t)\}}\frac{x_{i}(t)}{N_{i}}\\
&\gamma_{i,u(t)}=\sum_{S \subseteq \mathcal{V}: \; \{i,u(t)\} \subseteq S} (-1)^{|S|}\varphi_{S}\prod_{i \in S \setminus \{i,u(t)\}}\frac{x_{i}(t)}{N_{i}}
\end{align*}
Substituting the probability distribution from Corollary \ref{cor: drawprob} into equation \eqref{eq: probrate_restated}, we have
\begin{align*}
\frac{
	-\left(1/N_{i}\right)\gamma_{i} - \left(\frac{x_{u(t)}(t)}{N_{i}N_{u(t)}}\right) \gamma_{i,u(t)}
}
{
	-\left(1/N_{u(t)}\right)\gamma_{u(t)}-\left(\frac{x_{i}(t)}{N_{i}N_{u(t)}}\right)\gamma_{i,u(t)}
}
& \geq
\frac{
	-\left(1/N_{i}\right)\gamma_{i}  - \left(\frac{x_{u(t)}(t)+1}{N_{i}N_{u(t)}}\right) \gamma_{i,u(t)}
}
{
	-\left(1/N_{u(t)}\right)\gamma_{u(t)}-\left(\frac{x_{i}(t)}{N_{i}N_{u(t)}}\right)\gamma_{i,u(t)}
}
\\
- \left(\frac{x_{u(t)}(t)}{N_{i}N_{u(t)}}\right) \gamma_{i,u(t)}
& \geq
- \left(\frac{x_{u(t)}(t)+1}{N_{i}N_{u(t)}}\right) \gamma_{i,u(t)}
\\
0
& \leq
 \gamma_{i,u(t)},
\end{align*}
which is true from Lemma \ref{lem: problemma}.
\halmos
\end{proof}

\end{proof}

\section{Conclusion} \label{sec: conclusion} 

We have presented a multi-urn search problem as a model for searching for a specific vertex in a network.  Using this model, we have shown that there is always an optimal block policy in searches that meet the multi-urn search problem assumptions, irrespective of correlations in the probability model. We have also provided necessary and sufficient conditions for block policy optimality in two specific cases: independent urns and the single red marble scenario.   Finally, we gave a few properties of the dynamics of the multi-urn search problem and commented on the challenges of finding more general optimality conditions.

\subsection*{Future Work}

There are additional generalizations and extensions that we have not considered here, but which might also have interesting applications in modeling search.  One such generalization is removing the constraint that an urn can have at most one red marble.  This generalization might be an appropriate model for a network vertex search problem in which the network structure allowed for multiple edges between a pair of nodes.  Allowing for multiple red marbles in a single urn substantially changes the dynamics of the multi-urn search problem.

Another area of further inquiry could involve examination of the performance of different policies under various urn probability models.  We have shown, for example, that a purely greedy policy is not optimal in the case of independent urns.  However, the counter-example suggests that there could be some lower bounds on greedy policy performance, which might depend on the total number of urns and total number of marbles.  At the very least, there appears to be limits on how suboptimal we can make a greedy policy when constructing a two-urn data set.  Development in this direction could build on the results presented in \cite{chen2015sequential}.

We have assumed throughout our analysis that the target of the search would be easily identifiable to the searcher.  In our urn model this assumption translated to clear color distinction, so we assume we know immediately whether a drawn marble is red or blue.  However, we could relax this assumption in several ways.  We could, for example, characterize each marble with a feature set and develop a probability model that gives us the probability that a marble is a red marble, given its set of features.  Depending on the context of the search, this type of incomplete information model could evolve into a stopping problem in which the objective is to determine the best time to stop drawing new marbles in search of a red one.  Multi-arm bandit models and sequential mutual information maximization \cite{chen2015sequential} might offer useful approaches in this scenario.

A slight variation from this imperfect information approach would be to represent some marbles as being more ``red'' than others, according to some probability model.  In this model, the searcher receives a reward at the end of the search that is a function of the most ``red'' marble found, but still has to pay a fixed cost for each draw.  Like the imperfect information model, this formulation would ultimately be a stopping problem, balancing the current reward attained against the likelihood of attaining a higher reward by drawing more marbles.

We have also assumed uniform urn costs.  It is plausible, however, that in some cases it might cost more to draw marbles from some urns than from others.  Or, there could be a one-time access cost in order to gain the ability to draw marbles from the urn, which could represent law enforcement having to get a warrant to obtain internet or phone records for an individual.  

Our assumption that marbles are drawn in a random sequence from each urn might not be valid in some settings.  The Google search engine, for example, returns the most relevant results first, so that if a user does not find what he is looking for in the first few pages of results it might  make sense to try a different query rather than look through the remainder of the pages.  If we changed our model so that red marbles were more likely to be drawn first in each urn, then the problem would involve deciding when to stop querying an urn and switch to one that might be more promising.



\bibliography{urn-refs}

\begin{thebibliography}{15}
\providecommand{\natexlab}[1]{#1}
\providecommand{\url}[1]{\texttt{#1}}
\expandafter\ifx\csname urlstyle\endcsname\relax
  \providecommand{\doi}[1]{doi: #1}\else
  \providecommand{\doi}{doi: \begingroup \urlstyle{rm}\Url}\fi

\bibitem[Auer et~al.(2002)Auer, Cesa-Bianchi, and Fischer]{auer2002finite}
Peter Auer, Nicolo Cesa-Bianchi, and Paul Fischer.
\newblock Finite-time analysis of the multiarmed bandit problem.
\newblock \emph{Machine learning}, 47\penalty0 (2-3):\penalty0 235--256, 2002.

\bibitem[Berkhin(2005)]{pagerank-survey}
Pavel Berkhin.
\newblock A survey on pagerank computing.
\newblock \emph{Internet Mathematics}, 2\penalty0 (1):\penalty0 73--120, 2005.
\newblock \doi{10.1080/15427951.2005.10129098}.
\newblock URL \url{http://dx.doi.org/10.1080/15427951.2005.10129098}.

\bibitem[Bertsekas(2000)]{bertsekas-dp}
Dimitri~P. Bertsekas.
\newblock \emph{Dynamic Programming and Optimal Control}.
\newblock Athena Scientific, 2nd edition, 2000.
\newblock ISBN 1886529094.

\bibitem[Bubeck and Cesa{-}Bianchi(2012)]{bandit-survey}
S{\'{e}}bastien Bubeck and Nicol{\`{o}} Cesa{-}Bianchi.
\newblock Regret analysis of stochastic and nonstochastic multi-armed bandit
  problems.
\newblock \emph{CoRR}, abs/1204.5721, 2012.
\newblock URL \url{http://arxiv.org/abs/1204.5721}.

\bibitem[Chen et~al.(2015)Chen, Hassani, Karbasi, and
  Krause]{chen2015sequential}
Yuxin Chen, S~Hamed Hassani, Amin Karbasi, and Andreas Krause.
\newblock Sequential information maximization: When is greedy near-optimal?
\newblock In \emph{Proc. International Conference on Learning Theory (COLT)},
  2015.

\bibitem[Chung et~al.(2003)Chung, Handjani, and Jungreis]{Chung2003}
Fan Chung, Shirin Handjani, and Doug Jungreis.
\newblock Generalizations of polya's urn problem.
\newblock \emph{Annals of Combinatorics}, 7\penalty0 (2):\penalty0 141--153,
  2003.
\newblock ISSN 0219-3094.
\newblock \doi{10.1007/s00026-003-0178-y}.
\newblock URL \url{http://dx.doi.org/10.1007/s00026-003-0178-y}.

\bibitem[Downey et~al.(2006)Downey, Etzioni, and
  Soderland]{downey2006probabilistic}
Doug Downey, Oren Etzioni, and Stephen Soderland.
\newblock A probabilistic model of redundancy in information extraction.
\newblock Technical report, DTIC Document, 2006.

\bibitem[Frostig and Weiss(2016)]{frostig}
Esther Frostig and Gideon Weiss.
\newblock Four proofs of gittins' multiarmed bandit theorem.
\newblock \emph{Annals of Operations Research}, 241\penalty0 (1):\penalty0
  127--165, 2016.
\newblock ISSN 1572-9338.
\newblock \doi{10.1007/s10479-013-1523-0}.
\newblock URL \url{http://dx.doi.org/10.1007/s10479-013-1523-0}.

\bibitem[Gittins and Jones(1974)]{gittins1974}
JC~Gittins and DM~Jones.
\newblock A dynamic allocation index for new-product chemical research.
\newblock \emph{report), CUED/A-Mat Stud/TR13, Department of Engineering,
  Cambridge University}, 1974.

\bibitem[Gittins(1979)]{gittins1979bandit}
John~C Gittins.
\newblock Bandit processes and dynamic allocation indices.
\newblock \emph{Journal of the Royal Statistical Society. Series B
  (Methodological)}, pages 148--177, 1979.

\bibitem[Guynn and Weis(2016)]{usatoday}
Jessica Guynn and Elizabeth Weis.
\newblock Twitter suspends 125,000 {ISIL}-related accounts.
\newblock \emph{{USA} Today}, February 6, 2016.
\newblock URL
  \url{http://www.usatoday.com/story/tech/news/2016/02/05/twitter-suspends-125000-isil-related-accounts/79889892/}.
\newblock
  \url{http://www.usatoday.com/story/tech/news/2016/02/05/twitter-suspends-125000-isil-related-accounts/79889892/};
  Accessed April 12, 2016.

\bibitem[Lai and Robbins(1985)]{lai}
Tze~Leung Lai and Herbert Robbins.
\newblock Asymptotically efficient adaptive allocation rules.
\newblock \emph{Advances in applied mathematics}, 6\penalty0 (1):\penalty0
  4--22, 1985.

\bibitem[Levi et~al.(2016)Levi, Magnanti, and Shaposhnik]{levi2016scheduling}
Retsef Levi, Thomas Magnanti, and Yaron Shaposhnik.
\newblock Scheduling with testing.
\newblock Technical report, Working paper, 2016.

\bibitem[Mahmoud(2008)]{mahmoud2008polya}
H.~Mahmoud.
\newblock \emph{Polya Urn Models}.
\newblock Chapman \& Hall/CRC Texts in Statistical Science. CRC Press, 2008.
\newblock ISBN 9781420059847.
\newblock URL \url{https://books.google.com/books?id=7Bizo28c2LQC}.

\bibitem[Wei(1979)]{polya-medical}
L.~J. Wei.
\newblock The generalized polya's urn design for sequential medical trials.
\newblock \emph{The Annals of Statistics}, 7\penalty0 (2):\penalty0 291--296,
  1979.
\newblock ISSN 00905364.
\newblock URL \url{http://www.jstor.org/stable/2958811}.

\end{thebibliography}
\bibliographystyle{plainnat}

\end{document}